\documentclass[10pt,reqno]{amsart}
\usepackage{amsmath, latexsym, amsfonts, amssymb,amsthm, amscd,epsfig,enumerate}
\usepackage{subfigure,color, float, bm}

\advance\hoffset -.75cm

\usepackage{xcolor}

\definecolor{refkey}{gray}{.75}
\definecolor{labelkey}{gray}{.75}

\newcommand{\Z}{\mathbb Z}
\newcommand{\N}{\mathbb N}

\newcommand{\E}{\mathbb E}

\newcommand{\diff}{\mathrm{d}}

\newcommand{\pr}{\mathbb P}

\newcommand{\ident}{{\mathchoice {\rm 1\mskip-4mu l} {\rm 1\mskip-4mu l}
{\rm 1\mskip-4.5mu l} {\rm 1\mskip-5mu l}}}

\newcommand{\gegerm}{{\,\ge_{\textrm{germ}\,}}}

\newcommand{\gepgf}{{\,\ge_{\textrm{pgf}\,}}}

\newcommand{\boldmu}{{\bm{\mu}}}
\newcommand{\boldnu}{{\bm{\nu}}}

\newtheorem{teo}{Theorem}[section]
\newtheorem{lem}[teo]{Lemma}

\newtheorem{rem}[teo]{Remark}
\newtheorem{pro}[teo]{Proposition}
\newtheorem{defn}[teo]{Definition}
\newtheorem{exmp}[teo]{Example}

\newtheorem{assump}[teo]{Assumption}

\title
{Results on branching random walks and rumor processes via germ order}

\author[D.~Bertacchi]{Daniela Bertacchi}
\address{D.~Bertacchi, Dipartimento di Matematica e Applicazioni,
Universit\`a di Milano--Bicocca,
via Cozzi 53, 20125 Milano, Italy.}
\email{daniela.bertacchi\@@unimib.it}

\author[F.~Zucca]{Fabio Zucca}
\address{F.~Zucca, Dipartimento di Matematica,
Politecnico di Milano,
Piazza Leonardo da Vinci 32, 20133 Milano, Italy.}
\email{fabio.zucca\@@polimi.it}

\date{}

\begin{document}

\begin{abstract}
Germ order is a non-standard stochastic order defined through the comparison of the generating functions of the processes. This order was first introduced for branching random walks with a constant breeding law and independent dispersal of offspring, which are characterized by a one-dimensional generating function. In this work, we investigate the properties of the extension of this concept to processes characterized by a multidimensional generating function, such as general branching random walks and rumor processes. In particular, we use germ ordering to characterize the behavior of certain branching random walks and rumor processes with inhomogeneous breeding/transmitting laws.
\end{abstract}

\maketitle
\noindent {\bf Keywords}: branching random walk, generating function, fixed point, extinction probability vectors, germ order, pgf order, strong survival.

\noindent {\bf AMS subject classification}: 60J80.

\section{Introduction}
\label{sec:intro}

Germ order was introduced in \cite{cf:Hut2022} as a new stochastic order for branching random walks (BRWs).
The author considered branching random walks in which particles have the same offspring distribution, denoted by $\mu$, at every site, and their offspring are independently dispersed according to a transition matrix $P$.
If $P$ is fixed, the behavior of the BRW depends only on $\mu$. Two processes are compared using the probability
generating function  of their breeding law, in a neighborhood of $z=1$, namely $G_\mu(z)=\sum_n z^n\mu(n)$.
More precisely,  we say that $\mu\gegerm\nu$ when $G_\mu(z)\le G_\nu(z)$, for all $z\in [\delta,1]$.
It was proved in  \cite{cf:Hut2022} that if $\mu\gegerm\nu$, then the extinction of the $\mu$-process (in a set or in the whole space) implies 
the extinction of the $\nu$-process.
An interesting fact is that $\mu\gegerm\nu$ can hold  even if it is not possible to construct a coupling such that the $\mu$-process always and everywhere has at least as many particles as the $\nu$-process. In other words, the two processes may be germ-comparable, but not comparable according to the classical stochastic order.

The notion of germ order has been extended in \cite{cf:BZgerm} to general BRWs, that is, processes where the breeding mechanism is given by choosing not only the number of children, but also their locations, with a possibly non-independent choice
(see Definition~\ref{def:ordering}). This generalization includes the case studied in \cite{cf:Hut2022}, which we refer to as the independently diffusing and homogeneous case, as well as independently diffusing nonhomogeneous BRWs (where children are dispersed independently according to a transition matrix $P$, and the distribution of the total number of offspring depends on the site).
We emphasize that, in the general case, the total distribution at each site (i.e., the law governing the number of children generated), although still relevant, no longer contains all the information required to compare pairs of processes.

We note that, while for independently diffusing and homogeneous BRWs, one can just consider the one-dimensional generating function of the total distribution, in the general case we deal with a multidimensional generating function. Given two processes, the fact that the total distribution of the first, at each site, is germ-larger than the total distribution of the second, is necessary but not sufficient for germ domination between the two BRWs 
(see Theorem~\ref{pro:germindepdiff}). Also for the general case, if $\boldmu\gegerm\boldnu$, then the extinction  of the $\boldmu$-process implies the extinction of the $\boldnu$-process (see Theorem~\ref{th:germorder} for a more complete statement).

Moreover, we observe that germ order can be considered for all processes where there is a family of measures $\boldmu=\{\mu_y\}_{y \in Y}$, indexed by a 
set $Y$, with $\mu_y$ defined on the simultaneous choice of the total number of offspring and their types, in a set $X$ (see Section~\ref{sec:basic}
for details). This includes rumor processes where the rumor is transmitted on a graph thanks to random numbers and types of spreaders or listeners.

The aim of this paper is to investigate properties of germ order in this general framework, how it can be applied to compare different processes, and how
this comparison can help to overcome difficulties related to inhomogeneity of the environment.
Indeed, in a homogeneous setting, it is not difficult to characterize survival/extinction of the BRW: the total number of particles in the graph in 
this case is described by a branching process (there is extinction if and only if the expected number of offspring is smaller or equal to 1).
As soon as there are inhomogeneities (i.e., breeding mechanisms that depend on the site), the characterization of extinction becomes nontrivial and, in general, remains an open question. Even seemingly natural conjectures can be disproven. For instance, even if the average number of children at each site is greater than $1 + \varepsilon$ for some $\varepsilon > 0$, it is not always true that the BRW survives
(see \cite[Example 2.10]{cf:BRZ16}). 
The statement is true when there is some regularity. An example of this regularity is the quasi-transitive case
(see \cite[Section 2.4]{cf:BZ14-SLS} or \cite[Section 4.1]{cf:Z1} for details). A wider class of BRWs, which are sufficiently regular, so that we can characterize extinction/survival, are $\mathcal{F}$-BRWs, which roughly speaking have a finite number of types, ``nicely" dispersed on the graph
(see \cite{cf:BZ2} for the formal definition and Proposition~\ref{prop:fbrwgen} for a characterization).

When trying to establish whether a BRW $\boldnu$ goes extinct, one can compare it to another BRW $\boldmu$, whose behavior is already known.
In this sense, a classical tool is the construction of a coupling such that almost surely, at any time and everywhere, the $\boldnu$-process has
fewer particles than the $\boldmu$-process.
Then the extinction of the $\boldmu$-process implies the extinction of the $\boldnu$-process.
The construction of such a coupling is not always possible, but when $\boldmu\gegerm\boldnu$ (which is a weaker condition), we can
gather that the extinction of the $\boldmu$-process implies the extinction of the $\boldnu$-process (see Example~\ref{ex:fbrw}).

Thus germ order is a more powerful tool, with respect to the standard stochastic order, to study extinction/survival, through a comparison between
processes whose behavior is known and other whose behavior is unknown. It is crucial to understand how we can make use of this tool to compare 
homogeneous and inhomogeneous processes, and/or independently diffusing and non-independently diffusing processes.
In this work we show how one can use this technique in BRWs, in branching processes in varying environment (BPVEs) and in rumor processes.

Here is the outline of the paper. In Section~\ref{sec:basic} we give the definition of the multidimensional generating function of a general process and recall the notions of germ and probability generating function (briefly, pgf) orders. 
In Section~\ref{sec:germ} we study some properties of germ order. In particular, we give sufficient conditions for germ domination in the independently diffusing case (Theorems~\ref{pro:germindepdiff} and~\ref{pro:betterconditions}) and we compare processes with Poisson and mixed Poisson total distribution. We show that there are examples where the germ ordering is in one direction or the other, even if there is no classical coupling and if there is no ordering between the support of the parameters of the mixed case and the parameters of the classical Poisson (Examples~\ref{exmp:mix} and \ref{exmp:mix2}). 
We also address the question whether independent diffusion makes the process always germ-smaller or germ-larger than the non-independent case. 
The answer to both questions is negative: for example, placing all the offspring at a single, randomly chosen site results in a smaller process, while distributing the offspring evenly across the possible sites leads to a larger process (see Examples~\ref{exmp:nonind2} and \ref{exmp:nonind1}). This suggests that, in order to survive, the population should aim to spread out evenly. Note that the question of how the dispersal strategy affects population survival has been addressed in a slightly different model (where random catastrophes are present) in 
\cite{cf:disp3, cf:disp1, cf:disp2,  cf:dispSchi}.

In Section~\ref{sec:BRW} we give the definition of BRW and of survival and extinction of the process. We recall the notion of
$\mathcal F$-BRWs, which are a general family of (possibly inhomogeneous) BRWs, for which a characterization of global survival is known. 
We state the main result, which allows to use germ coupling 
to understand whether a BRW survives or goes extinct (Theorem~\ref{th:germorder}). This statement proves to be particularly useful when 
the process is inhomogeneous and a classic stochastic coupling is not available, as shown in Examples~\ref{ex:fbrw} and \ref{ex:fbrw2}.

Section~\ref{sec:BPVE} is devoted to a particular family of BRWs on $\N$, branching processes in varying environment, or BPVE.
These processes can be seen as branching processes, where reproduction laws depend on the generation. We show, in Example~\ref{exmp:BPVE0},
a BPVE for which
the expected number of offspring in each generation is smaller than one, and the process survives with positive probability.
We use germ order to compare a general BPVE with this example and give a simple sufficient condition for survival (Example~\ref{exmp:BPVE1}).
This condition is generalized in Proposition~\ref{pro:BPVE1}. We also use germ order to prove Theorem~\ref{pro:bertrodzuc}, which is similar to \cite[Theorem 2.6]{cf:BRZ16}, but additionally applies to processes with breeding laws that do not necessarily have all finite moments.

In Section~\ref{sec:firework} we study rumor processes on $\N$, namely firework and reverse firework processes. 
These were first introduced and studied on $\N$
in \cite{cf:JMZ} and later, on $\N$ and on Galton-Watson trees, in \cite{cf:BZ13}. In both processes, each site of a graph contains a random number of stations: emitters in the case of the firework process and listeners in the case of the reverse firework process. Each station has a random radius, representing the maximum distance its signal can reach in the first process, and the maximum distance from which it can receive signals in the second process. The question is whether the process dies out (i.e. almost surely
only a finite number of vertices are informed) or it survives (i.e. with positive probability infinitely many vertices are informed).
In Section~\ref{subsec:mainrumor} we generalize the results of \cite{cf:BZ13} to the multi-type case, in which at each site there might
be different types of stations, each equipped with its own radius law. In this general case, the number of stations at each site is prescribed by
a family of measures $\boldmu$.
We provide characterizations and/or sufficient conditions for
survival and extinction (Theorems~\ref{thm:2}, \ref{th:reversehomogeneous}, \ref{thm:3}, \ref{th:reverseinhomogeneous}).
We prove that, if $\boldmu\gegerm\boldnu$, when the rumor process associated with $\boldnu$ survives with positive probability, so does the one associated with $\boldmu$ (Theorems~\ref{pro:comparehomfirework}, \ref{pro:comparereverse}).
We conclude the section with three examples which show counterintuitive behavior. In Example~\ref{exmp:fireworkmixedpoisson1} a firework process
where the number of stations is a mixed Poisson, we show that, even if there are arbitrarily long stretches of sites where the number is subcritical,
nevertheless we have survival. In Example~\ref{ex:nonid} we have a rumor process with two types of emitters. 
The number of emitters is
Poisson distributed. The two types differ in the radius law, one of which is subcritical and the other is supercritical.
If the type is randomly assigned to be the same for all the emitters at a site, then the process dies out (even if there are arbitrarily long stretches 
with emitters of supercritical type). 
If the type is assigned randomly to at most one emitter, while the others are half of one type and half of the other, then the process survives
(even if there are arbitrarily long stretches with emitters of subcritical type). 
Example~\ref{ex:nonidMP} is similar, with a mixed Poisson instead of the Poisson.

Section~\ref{sec:proofs} contains the proofs of the statements of the previous sections.

\section{Basic definitions and properties}
\label{sec:basic}

\subsection{Generating functions}\label{subsec:process}

Given an at most countable set $X$ and a set $Y$,
we consider a family of measures $\boldmu=\{\mu_y\}_{y \in Y}$
defined on the (countable) measurable space $(S_X,2^{S_X})$
where $S_X:=\{f:X \to \N\colon |f|<+\infty\}$ where $|f|:=\sum_{x\in X}f(x)$. 

One may think of $Y$ as the space set and of $X$ as the type set: at site $y$ we pick a function $f$ with distribution $\mu_y$, 
and we generate $f(x)$ individuals of type $x$, for all $x\in X$. 
Although this is not the only possible interpretation, we stick to
this wording for most of the paper.

To the family $\boldmu$
we associate the following \textit{generating function} $G_\boldmu:[0,1]^X \to [0,1]^Y$,
\begin{equation}\label{eq:genfun}
G_\boldmu({\mathbf{z}}|y):= \sum_{f \in S_X} \mu_y(f) \prod_{x \in X} {\mathbf{z}}(x)^{f(x)}, 
\end{equation}
where $G_\boldmu({\mathbf{z}}|y)$ is the $y$ coordinate of $G_\boldmu({\mathbf{z}})$.
The family $\boldmu$
is uniquely determined by $G_\boldmu$ (see for instance \cite[Section 2.3]{cf:BZ14-SLS} or \cite[Section 2.2]{cf:BZ2017} and \cite[proof of Proposition 2.1]{cf:BZgerm}.
Henceforth, when possible, we write $G$ instead of $G_\boldmu$.
We denote by  
\[
\phi_y(t) := G_\boldmu(t\mathbf{1}|y)=\sum_{f \in S_X} \mu_y(f) t^{|f|},
\]
where 
$\mathbf{1}\in [0,1]^X$ 
is such that $\mathbf{1}(x)=1$ for all $x\in X$.
Similarly, we denote by $\mathbf{0}\in[0,1]^X$ the vector $\mathbf{0}(x)=0$ for all $x\in X$.
Note that, if
\begin{equation}\label{eq:total}
    \rho_y(n):= \mu_y(f\colon |f|=n),
\end{equation}
then $\phi_y$ is the one-dimensional generating function of $\rho_y$.
We refer to $\{\rho_y\}_{y \in Y}$ as the \textit{total} distribution of $\boldmu$ (or
\textit{total offspring} distribution).
If we need to stress the dependence on $\boldmu$, we write $\rho_y^\boldmu$.
Note that, if $T\sim\rho_y$, then $\E[T]=\frac{\diff}{\diff t}\phi_y(1)$.
The topological properties of $G_\boldmu$ are described in the following proposition; in particular, we define $\|\mathbf{z}\|_\infty:=\sup_{x \in C} |\mathbf{z}(x)|$ the restriction of the norm of $l^\infty(C)$ to $[0,1]^C$ (for $C\in\{X,Y\}$). The (partially ordered) spaces $[0,1]^X$ and $[0,1]^Y$ can be equipped with two useful topologies: the product (or \textit{pointwise convergence}) topology  and the finer topology arising from the metric $d(\mathbf{z}, \mathbf{v}):=\|\mathbf{z}- \mathbf{v}\|_\infty$ .

\begin{pro}\label{pro:Gtopology}
Let us consider the generating function $G_\boldmu$ defined by eq.~\eqref{eq:genfun}.
\begin{enumerate}
\item $G$ is 
non-decreasing
with respect to the usual partial order of $[0,1]^X$ and $[0,1]^Y$.
\item $G$ is continuous with respect to the \textit{pointwise convergence topology} of $[0,1]^X$ and $[0,1]^Y$.
\item If the family $\{\rho_y\}_{y \in Y}$ is tight then $G$ is uniformly continuous with respect to the $\|\cdot\|_\infty$-topologies of  $[0,1]^X$ and $[0,1]^Y$.
\end{enumerate}
\end{pro}

Given a family $\{\rho_y\}_{y \in Y}$ of measures on $\mathbb{N}$ and a non-negative stochastic matrix $P=(p(y,x))_{y \in Y, x \in X}$ (clearly $\sum_{x \in X} p(y,x)=1$ for all $y \in Y$), one can construct a particular family $\boldmu$, with total offspring distribution $\{\rho_y\}_{y \in Y}$.

We say that $\boldmu$ is an \textit{independently diffusing} family of measures if
\begin{equation}\label{eq:particular1}
	\mu_y(f)=\rho_y 
        \left (|f|\right )\frac{(|f|)!}{\prod_{x \in X} f(x)!}
     \prod_{x \in X} p(y,x)^{f(x)},
	\quad \forall f \in S_X.
\end{equation}
According to this family $\boldmu$, at $y$ we generate $n$ individuals with probability $\rho_y(n)$ and we independently assign to each one, a type 
$x$ with probability $p(y,x)$.
It is easy to prove that for an independently diffusing family $\boldmu$, we have that
\begin{equation}\label{eq:Gindepdiff}
 G(\mathbf{z}|y)=\phi_y (P\mathbf{z}(y)), \quad \forall y \in Y, \, \mathbf{z} \in [0,1]^X,
\end{equation}
where 
$P\mathbf{z}(y)=\sum_{x \in X} p(y,x)\mathbf{z}(x)$ (explicit computations can be found in \cite{cf:BZgerm}).

Particular choices of the sets $X$ and $Y$, and of the family $\boldmu$ give rise to known processes.
In the case of branching random walks (or multi-type branching processes), $X=Y$ represents space where particles live
(or the type set), and when a particle breeds, it does so according to a law depending on its location $y$, which picks the function $f$, and places $f(x)$ children at site $x$, for all $x\in X$
(or assigns type $x$ to $f(x)$ children).

In the firework (or reverse firework) processes, $Y$ is the space where stations are placed, and $X$ is the set of the possible different types of
stations. The configuration of the station is random, in that at site $y$ there are $f(x)$ stations of type $x$, for all $x\in X$, with probability 
$\mu_y(f)$.

\subsection{Germ order}\label{sec:germ-def}
By a comparison between generating functions, we introduce the pgf and the germ order relations on the families of measures on $\mathbb{R}^X$ with support on $S_X$, indexed by $Y$.
The definition of germ order generalizes the one given in \cite{cf:Hut2022} and has been originally studied in \cite{cf:BZgerm}.
 
\begin{defn}\label{def:ordering}
 Let $\boldmu:=\{\mu_y\}_{y \in Y}$ and $\boldnu:=\{\nu_y\}_{y \in Y}$ be two families of measures on $\mathbb{R}^X$ with support on $S_X$.
 Let $G_{\boldmu}$ and $G_\boldnu$ be the corresponding generating functions.
 \begin{enumerate}
  \item $\boldmu \gepgf \boldnu$ if and only if
  $G_\boldmu(\mathbf{z}) \le G_\boldnu(\mathbf{z})$ for all $\mathbf{z} \in [0,1]^X$.
  \item $\boldmu \gegerm \boldnu$ if and only if there exists $\delta \in [0,1)$
  $G_\boldmu(\mathbf{z}) \le G_\boldnu(\mathbf{z})$ for all $\mathbf{z} \in [\delta,1]^X$.
 \end{enumerate}
 When, given two single measures $\mu$ and $\nu$, we write $\mu \gepgf \nu$ (resp.~$\mu \gegerm \nu$), we mean $\{\mu\} \gepgf  \{\nu\}$ (resp.~$\{\mu\}\gegerm \{\nu\}$).
 \end{defn}
 
 We observe that $\boldmu \gepgf \boldnu \ \Longrightarrow \ \boldmu \gegerm \boldnu$ (just take $\delta=0$), but the reverse implication does not hold. 
Clearly $G_\boldmu(\mathbf{z}) \le G_\boldnu(\mathbf{z})$ if and only if $G_\boldmu(\mathbf{z}|y) \le G_\boldnu(\mathbf{z}|y)$ for all $y \in Y$; thus,
$\boldmu \gegerm \boldnu$ (with a certain $\delta <1$) if and only if $\mu_y \gegerm \nu_y$ for all $y \in Y$ 
(with the same $\delta <1$). Similarly, 
$\boldmu \gepgf \boldnu$ if and only if $\mu_y \gepgf \nu_y$ for all $y \in Y$. 
 The binary relation $\gegerm$ is a partial order as proven in \cite[Proposition 2.1]{cf:BZgerm}.

\section{Results on germ order}\label{sec:germ}

The following result links germ order (and pgf order, when $\delta=0$) with the order of the corresponding families $\phi_y$.
It appeared as Proposition 2.3 in \cite{cf:BZgerm}, we recall it here for completeness.

\begin{teo}\label{pro:germindepdiff}
	Suppose that $\boldmu$ and $\boldnu$ are two families of measures. 
		Consider the following for any fixed $\delta<1$:
	\begin{enumerate}
		\item $G_\boldmu(\mathbf{z}) \le G_\boldnu(\mathbf{z})$ for all $\mathbf{z} \in [\delta,1]^X$;
		\item $\phi_y^\boldmu(t) \le \phi_y^\boldnu(t)$ for all $t \in [\delta,1]$ and all $y \in Y$.
	\end{enumerate}
	Then $(1) \Rightarrow (2)$. Moreover if $\boldmu$ and $\boldnu$ are independently diffusing families with the same matrix $P$ (see Equations~\eqref{eq:particular1} and \eqref{eq:Gindepdiff}), 
	then $(2) \Rightarrow (1)$.	
\end{teo}
Roughly speaking, condition $\rho^\boldmu_y \gegerm \rho_y^\boldnu$ for all $y \in Y$ with a fixed $\delta<1$ 
is always necessary for $\boldmu \gegerm \boldnu$, and is equivalent to $\boldmu \gegerm \boldnu$, when 
$\boldmu$ and $\boldnu$ are independent diffusing families with the same matrix $P$.
One has to be careful since condition (2) in Theorem~\ref{pro:germindepdiff} must hold for the same $\delta$ for all $y\in Y$.
This leads to  two relevant particular cases, when $\boldmu$ are independent diffusing families with the same matrix $P$.
On the one hand, $\boldmu \gepgf \boldnu$ if and only if 
$\rho^\boldmu_y \gepgf \rho_y^\boldnu$ for all $y \in Y$.
On the other hand, if the total distributions do not depend on $y$, that is, $\rho^\boldmu_y=\rho^\boldmu$ and $\rho^\boldnu_y=\rho^\boldnu$,
then $\phi^\boldmu(t)\le\phi^\boldnu(t)$ for all $t\in [\delta,1]$, for some $\delta\in[0,1)$, is equivalent to $\boldmu\gegerm \boldnu$.
This is the case treated in \cite{cf:Hut2022}, where criteria based on the moments of
$\rho^\boldmu$ and $\rho^\boldnu$ were stated. For instance, it suffices that $\E[T_\boldmu]>\E[T_\boldnu]$, given two random variables $T_\boldmu\sim \rho^\boldmu$ and
$T_\boldnu\sim \rho^\boldnu$. 
When the total distributions depend on $y$, it is not straightforward to find simple criteria based on the moments of
$\rho^\boldmu_y$ and $\rho^\boldnu_y$.

\subsection{Sufficient conditions for germ order of independently diffusing families}

We want to provide sufficient conditions for $\boldmu \gegerm \boldnu$, in the case where the two families are independently diffusing, with the same matrix $P$. 
Note that if $\phi_y^\boldmu(0)=\rho^\boldmu_y(0)>\rho^\boldnu_y(0)=\phi_y^\boldnu(0)$, then $\boldmu\not\!\!\gepgf\boldnu$ (apply Theorem~\ref{pro:germindepdiff} with $\delta=0$).

\begin{teo}\label{pro:germidextended}
	Suppose that $\boldmu$ and $\boldnu$ are two families of independently diffusing measures, with the same matrix $P$. 
	Let $U_y$ and $W_y$ be distributed as the total distributions of $\mu_y$ and $\nu_y$ respectively.
	\begin{enumerate}
		\item
		If $\E[U^k_y] \in (\E[W^k_y], +\infty]$, then there exists $\delta=\delta(y)>0$ such that  $\E[U^k_y \exp(-t U_y)]\ge\E[W_y]$, for all $t\in [0,\delta]$.
		\item
		Suppose that $\E[U_y^i]=\E[W_y^i]$ for all $i =0, \ldots, k$ (for some $k \in \mathbb{N}$) and there exists $ \delta>0$ such that 
		\begin{equation}\label{eq:condvarphi}
		  \begin{cases}
			\E[U_y^{k+1} \exp(-\delta U_y)]\ge \E[W_y^{k+1}] & \text{ if $k$ is even}\\
			\E[W_y^{k+1} \exp(-\delta W_y)]\ge \E[U_y^{k+1}] & \text{ if $k$ is odd}\\
		\end{cases}  
		\end{equation}
				for all $y\in Y$.
		Then $\boldmu\gegerm\boldnu$.
	\end{enumerate}
\end{teo}

In Theorem~\ref{pro:germidextended}	we compare the functions $\phi_y^\boldmu(\exp(-t))$ and $\phi_y^\boldnu(\exp(-t))$ instead of the generating functions $\phi_y^\boldmu(t)$ and $\phi_y^\boldnu(t)$. Indeed, it is trivial to note that there exists $\delta_1<1$ such that $\phi_y^\boldmu(t) \le \phi_y^\boldnu(t)$ for all $t \in [\delta_1,1]$ if and only if there exists $\delta>0$ such that $\phi_y^\boldmu(\exp(-t)) \le \phi_y^\boldnu(\exp(-t))$ for all $t \in [0,\delta]$. A similar 
condition for $\boldmu\gegerm\boldnu$ is given by the following proposition.

\begin{teo}\label{pro:betterconditions}
  Let $\boldmu$ and $\boldnu$ be two families of independently diffusing measures, with the same matrix $P$. 
	Let $U_y$ and $W_y$ be distributed as the total distributions of $\mu_y$ and $\nu_y$ respectively.
  Suppose that $\E[\prod_{j=0}^{i-1} (U_y-j)]=\E[\prod_{j=0}^{i-1} (W_y-j)]$ for all $i =1, \ldots, k$ (for some $k \in \mathbb{N}$, if $k=0$ the previous condition is trivial) and there exists $ \delta_1>0$ such that 
\begin{equation}\label{eq:condphi}
   \begin{cases}
	\E[\delta_1^{U_y-k-1}\prod_{j=0}^{k} (U_y-j)]\ge \E[\prod_{j=0}^{k} (W_y-j)] & \text{ if $k$ is even}\\
	\E[\delta_1^{W_y-k-1}\prod_{j=0}^{k} (W_y-j) ]\ge \E[\prod_{j=0}^{k} (U_y-j)] & \text{ if $k$ is odd}\\
\end{cases} 
\end{equation}
for all $y\in Y$. Then $\boldmu\gegerm\boldnu$.
\end{teo}

 We observe that \eqref{eq:condphi} implies \eqref{eq:condvarphi}, while the converse does not hold.
 More details can be found in Section~\ref{sec:proofs}.
  Note that \eqref{eq:condphi}, with $k=0$, are the conditions which appear in the statement of Proposition~\ref{pro:comparemix}.

\subsection{A comparison between independently diffusing families with mixed Poisson total distribution}\label{subsec:excompare}

We now compare two independently diffusing families $\boldmu$ and $\boldnu$, who share the same matrix $P$
and have Poisson total distributions.

It is trivial to show that if $\rho^\boldmu_y \sim \mathcal{P}oi(\lambda_1)$ and $\rho^\boldnu_y \sim \mathcal{P}oi(\lambda_2)$, for all $y\in Y$, 
then the following are equivalent: $\lambda_1 \ge \lambda_2$, $\boldmu \gegerm \boldnu$, $\boldmu \gepgf \boldnu$, $\boldmu \succeq \boldnu$ (where 
$\succeq$ is the usual stochastic domination).
In the following example, we compare $\boldmu$ and $\boldnu$ which have inhomogeneous total distributions. With a little more effort one could prove that the stronger inequality $\boldmu \succeq \boldnu$ holds.
\begin{exmp}\label{ex:inhpoi}
  Let $\boldmu$ and $\boldnu$ be two independently diffusing families of measures, with the same matrix $P$ and with total distributions $\rho^\boldmu_y\sim \mathcal{P}oi(\lambda_y^\boldmu)$, 
  $\rho^\boldnu_y\sim \mathcal{P}oi(\lambda_y^\boldnu)$,  where $\lambda_y^\boldmu\ge\lambda_y^\boldnu>0$, for all $y\in Y$.
Then $\boldmu\gepgf\boldnu$.   This follows easily by applying Theorem~\ref{pro:germindepdiff} to the explicit expressions of the generating functions $\phi^\boldmu_y(t)=\exp(\lambda^\boldmu_y(t-1))$ and $\phi^\boldnu_y(t)=\exp(\lambda^\boldnu_y(t-1))$.
\end{exmp}

This stochastic/pgf domination among families of measures with Poisson total distribution (where one varies the parameters, allowing random variables in the role of the parameter, as in mixed Poisson), can be mimicked in other distributions such as the geometric law, for instance. Therefore, what we are about to do in this section could be done, in principle, for the geometric distribution as well. Our choice of mixed Poisson laws over mixed geometric laws is purely pragmatic, since it simplifies the explicit computations.

Suppose now that the total distributions are mixed Poisson, that is, there exist
two families of $(0,+\infty)$-valued random variables
$\{\Lambda_{1,y}\}_{y\in Y}$ and   $\{\Lambda_{2,y}\}_{y\in Y}$, such that
\[\begin{split}
  \rho^\boldmu_y(j)&=\E[\exp(-\Lambda_{1,y})\cdot \Lambda_{1,y}^j/j!],\qquad \forall j\in\N;\\
  \rho^\boldnu_y(j)&=\E[\exp(-\Lambda_{2,y})\cdot \Lambda_{2,y}^j/j!],\qquad \forall j\in\N.
\end{split}
\]
This means that  the total distribution of the $\boldmu$ process, at $y$, is obtained by picking at random a parameter $\lambda\sim 
\Lambda_{1,y}$, and then generating a $\mathcal{P}oi(\lambda)$ random variable (and analogously for the number $\boldnu$).

The following proposition gives a sufficient condition for $\boldmu\gegerm\boldnu$.
\begin{pro}\label{pro:comparemix}
	Suppose that $\boldmu$ and $\boldnu$ are two families of independently diffusing families of measures, with the same matrix $P$. 
	Let
	 $\rho_y^\boldmu\sim\mathcal \mathcal{P}oi(\Lambda_{1,y})$ and $\rho_y^\boldnu\sim\mathcal \mathcal{P}oi(\Lambda_{2,y})$.
	\begin{enumerate}
		\item
		If $\E[\Lambda_{1,y}] \in (\E[\Lambda_{2,y}], +\infty]$, then there exists $\delta=\delta(y)<1$
		 such that  $\E[\Lambda_{1,y}\cdot\exp((t-1)\Lambda_{1,y})]\ge \E[\Lambda_{2,y}]$, for all $t\in [\delta,1]$.
		\item
		Suppose that there exists $ \delta<1$ such that 
		$\E[\Lambda_{1,y}\cdot\exp((\delta-1)\Lambda_{1,y})]\ge \E[\Lambda_{2,y}]$,  for all $y\in Y$.
		Then $\boldmu\gegerm\boldnu$.
	\end{enumerate}
\end{pro}
Note that if $\E[\exp(-\Lambda_{1,y})]>\E[\exp(-\Lambda_{2,y})]$, then $\boldmu \not \!\! \gepgf \boldnu$.

One may wonder if it is possible to compare a family $\boldnu$ with $\rho^\boldnu\sim \mathcal{P}oi(\lambda)$ ($\lambda$ being a positive real number),
to a family $\boldmu$ where $\rho^\boldmu_y$ is a mixed Poisson $\mathcal{P}oi(\Lambda_y)$.
The interesting case is when $\Lambda_y$ has support which intersects both $(\lambda,+\infty)$ and $(0,\lambda)$.
The following two examples show that,
both in the homogeneous and nonhomogeneous case, where 
$\Lambda_{y}\equiv\Lambda$  (i.e. the distributions do not depend on $y$), and the inhomogeneous one, 
it may happen that $\boldmu\gegerm\boldnu$ or $\boldnu \gegerm \boldmu$.


\begin{exmp}\label{exmp:mix}
Fix $\lambda, \varepsilon >0$ and $\alpha\in(0,1/2)$.
Let
$\Lambda\sim\alpha \delta_{\lambda-\varepsilon}+(1-\alpha) \delta_{\lambda+\varepsilon}$.
Let $\boldmu$ and $\boldnu$ be two independently diffusing families of measures, with the same matrix $P$ and with total distributions $\rho^\boldmu_y\sim \mathcal{P}oi(\Lambda)$, $\rho^\boldnu_y\sim \mathcal{P}oi(\lambda)$, for all $y\in Y$.
Then $\boldmu\gegerm\boldnu$. 
Moreover, if
$\alpha \in \big ( (1-\exp(-\varepsilon))/(\exp(\varepsilon)-\exp(-\varepsilon)), 1/2 \big )$, then
 $\boldmu \not \!\!\! \gepgf \boldnu$.
Indeed, it is enough to apply Proposition~\ref{pro:comparemix}, after
noting that
$\E[\Lambda] > \lambda$ and $\E[\exp(-\Lambda)]>\exp(-\lambda)$.

More generally, let $\Lambda_{1}\sim\alpha \delta_{\lambda_1}+(1-\alpha) \delta_{\lambda_2}$ and ${\Lambda_2}\sim\delta_\lambda$ where $\lambda_1 \le \lambda_2$. Clearly if $\lambda_1=\lambda_2 \ge \lambda$ (resp.~$\le \lambda$) then 
$\boldmu $ stochastically dominates $\boldnu$ (resp. $\boldnu $ stochastically dominates $\boldmu$) and this implies
$\boldmu\gegerm\boldnu$ (resp. $\boldnu\gegerm\boldmu$).
Suppose now that $\lambda_1 < \lambda_2$; we have that
\[
\begin{cases}
    \lambda_2 > 
    \lambda, \, \alpha < (\lambda_2-\lambda)/(\lambda_2-\lambda_1) \Longrightarrow \boldmu \gegerm \boldnu \\
    \lambda_1 < 
    \lambda, \, \alpha > (\lambda_2-\lambda)/(\lambda_2-\lambda_1) \Longrightarrow \boldnu \gegerm \boldmu. \\
\end{cases}
\]
\end{exmp}

It is not difficult to extend the same ideas to construct an example, where the parameters of the Poisson random variables depend on $y$.
The reader can find the computational details in Section~\ref{sec:proofs}.
\begin{exmp}\label{exmp:mix2}
Fix $0<\varepsilon <M<+\infty$ and $\{\lambda_y\}_{y\in Y}$, such that $\varepsilon<\lambda_y \le M$ for all $y\in Y$.
Let $\alpha_y$ such that $\sup_y \alpha_y<1/2$.
Let
$\Lambda_{y}\sim\alpha_y \delta_{\lambda_y-\varepsilon}+(1-\alpha_y) \delta_{\lambda_y+\varepsilon}$, 
for all $y\in Y$.
Let $\boldmu$ and $\boldnu$ be two independently diffusing families of measures, with the same matrix $P$ and with total distributions 
$\rho^\boldmu_y\sim \mathcal{P}oi(\Lambda_{y})$, 
$\rho^\boldnu_y \sim \mathcal{P}oi(\lambda_{y})$, 
respectively for all $y \in Y$.
Then $\boldmu\gegerm\boldnu$. Moreover, if $\alpha_y>(1-\exp(-\varepsilon))/(\exp(\varepsilon)-\exp(-\varepsilon))$, for some $y\in Y$, then $\boldmu \not \!\!\! \gepgf \boldnu$.

Conversely, let $\varepsilon$, $M$ and $\{\lambda_y\}_{y \in Y}$ as before and suppose that $\inf _y \alpha_y >1/2$. Then $\boldnu \gegerm \boldmu$. 
\end{exmp}
By putting together Examples~\ref{ex:inhpoi} and \ref{exmp:mix2}, we could compare two independently diffusing measures
$\boldmu$ and $\boldnu$ such that their total distributions are $\mathcal{P}oi(\Lambda_{1,y})$ and $\mathcal{P}oi(\Lambda_{2,y})$, respectively, where $\Lambda_{1,y}\sim\alpha_y \delta_{\lambda_y^\boldmu-\varepsilon_1}+(1-\alpha_y) \delta_{\lambda_y^\boldmu+\varepsilon_1}$ and  $\Lambda_{2,y}\sim\beta_y \delta_{\lambda_y^\boldnu-\varepsilon_2}+(1-\beta_y) \delta_{\lambda_y^\boldnu+\varepsilon_2}$.
Indeed, if $\lambda_y^\boldmu\ge\lambda_y^\boldnu$ and $\sup_y\alpha_y<1/2<\inf_y\beta_y$, then $\boldmu\gegerm\boldnu$.

\subsection{A comparison between independent and non-independent diffusion}
\label{subsec:genveindep}

Given an independently diffusing family $\boldnu$, one may wonder whether placing the same amount of individuals in 
different (non-independent) ways, provides a family of measures which is larger or smaller than $\boldnu$, according to
the $\gegerm$ order.\\
More precisely, let $\boldnu$ be an independently diffusing family whose total distribution is $\{\rho_y\}_{y\in Y}$:
we provide two families $\boldmu_1$ and $\boldmu_2$, which have the same total distribution as $\boldnu$,
and  such that $\boldmu_1 \gepgf \boldnu  \gepgf \boldmu_2$ (and thus
$\boldmu_1 \gegerm \boldnu  \gegerm \boldmu_2$).

The following example shows that assigning the same type (although randomly) to all the offspring, gives rise to a family of measures which is smaller 
than the independently diffusing one.
\begin{exmp}\label{exmp:nonind2}
Let $\boldnu$ be an independently diffusing family, with transition matrix $P=(p(y,x))_{y \in Y, x \in X}$. Let us define $\boldmu_2=\{\mu_{2,y}\}_{y\in Y}$ as follows
\[
\mu_{2,y}(f)=
\begin{cases}
    p(y,x) \rho_y^\boldnu(k) & \textrm{if } f=k \ident_{\{x\}}, \, k \in \N, \, x \in X \\
    0 & \textrm{otherwise},
\end{cases}
\]
where $\ident_A$ is the characteristic function of $A \subseteq X$.
Roughly speaking, according to $\mu_{2,y}$, we generate $k$ particles with probability $\rho^\boldnu_y(k)$ and assign to all of them the type $x$, with probability $p(y,x)$.
 Then the total offspring distribution of $\rho_y^{\boldmu_2}$ coincides with $\rho^\boldnu_y$, for all $y\in Y$, and $\boldnu\gepgf\boldmu_2$.

Indeed, let $\phi_y$ be the generating function of $\rho_y$.
We know that $G_\boldnu( \mathbf{z} |y):= \phi_y( P \mathbf{z}(y))$ for all $\mathbf{z} \in [0,1]^X$.
 The generating function of $\boldmu_2$ is $G_{\boldmu_2}(\mathbf{z}|y):= \sum_{x \in X} p(y,x) \phi_y(\mathbf{z}(x))$ for all $\mathbf{z} \in [0,1]^X$.
Since $\phi_y$ is a convex function,  for all $\mathbf{z} \in [0,1]^X$,
\[
G_\boldnu( \mathbf{z} |y) =\phi_y \Big ( \sum_{x \in X} p(y,x)\mathbf{z}(x) \Big ) \le 
\sum_{x \in X} p(y,x)\phi_y \big ( \mathbf{z}(x) \big )= G_{\boldmu_2}(\mathbf{z}|y).
\]
This proves that $\boldnu  \gepgf \boldmu_2$, thus $\boldnu  \gegerm \boldmu_2$.
\end{exmp}
The following example shows a non-independently diffusing family of measures which is larger than an independently diffusing one with the same total distribution. 
\begin{exmp}\label{exmp:nonind1}
Let $\boldnu$ be an independently diffusing family, with $|X|=n$ and transition matrix $P=(p(y,x))_{y \in Y, x \in X}$. Suppose that 
$p(y,x):=p_x \in \mathbb{Q}$ for all $x \in X, y \in Y$.

Let $k \in \N$ be such that $kp_x \in \N$ for all $x \in X$. An explicit example is $p_x=1/n$ and $k=n$. Choose a sequence of natural numbers $\{\alpha_i\}_{i \in \N}$ such that $k \alpha_i \le i$ and define $r_i:= i-k \alpha_i$. 
Define $\boldmu_1$ by the generating function $G_{\boldmu_1}(\mathbf{z}|y):=\sum_{i=0}^\infty 
\rho_y^\boldnu(i) \Big ( \prod_{x \in X} \mathbf{z}(x)^{\alpha_i k p_x}  \Big ) \Big ( \sum_{x \in X} p_x \mathbf{z}(x)  \Big )^{r_i}.$ 
Roughly speaking, this represents a process where a random number $i$ of particles is generated: $\alpha_i k p_x$ particle are chosen and labeled $x$ (and this is done for each $x \in X$); the remaining $r_i$ particles are assigned a random label independently according to the distribution $\{p_x\}_{x \in X}$. Then we have that $\rho^{\boldmu_1}_y\equiv\rho^\boldnu_y$, for all $y\in Y$, and $\boldmu_1\gepgf\boldnu$ (for the proof see Section \ref{sec:proofs}).
\end{exmp}

\section{Branching Random Walks and germ order}
\label{sec:BRW}

Given an at most countable set $X$ and a family of probability measures $\boldmu:=\{\mu_x\}_{x \in X}$ on $S_X$, the associated BRW, which we denote by
$(X,\boldmu)$ is a discrete-time stochastic process $\{\eta_n\}_{n \in \mathbb{N}}$ (time representing generations).

The dynamics is as follows: a particle at $x$ picks a random $f \in S_X$ according to $\mu_x$ then it dies and is replaced by $f(y)$ particles at $y$ (for all $y \in X$). In general, the children of a particle do not need to be placed independently, unless the family $\boldmu$ is independently diffusing: in this case we call the process \textit{independently diffusing BRW} or \textit{BRW with independent diffusion}. 
We say that the BRW starts at $x$ if $\eta_0=\delta_x$.
A formal construction of the process can be found in \cite[Section 2.1]{cf:Z1}

\begin{defn}\label{def:extinction-event} $\ $
We call \textsl{survival in $A \subseteq X$} the event 
\[
\mathcal S(A):=\Big \{\limsup_{n \to +\infty} \sum_{y \in A} \eta_n(y)>0\Big \},
\]
and we denote by $\mathcal E(A)=\mathcal S(A)^\complement$ the event that we call \textsl{extinction in $A$}. We denote by $\mathbf{q}(x,A)$ the probability of $\mathcal{E}(A)$ when the initial condition is 1 particle at $x$. We define the extinction probability vector $\mathbf{q}(A) \in[0,1]^X$ by its coordinates $\mathbf{q}(x,A)$ for $x \in X$.
\end{defn}

\begin{defn}\label{def:survival} $\ $
\begin{enumerate}
 \item 
The process \textsl{survives in $A \subseteq X$}, or there is survival in $A$, starting from $x \in X$, if $\mathbf{q}(x,A)<1$, 
otherwise the process \textsl{goes extinct in }$A$ (or dies out in $A$). 
\item
The process \textsl{survives globally}, starting from $x$, if
it survives in $X$.
 \item
 There is \textsl{strong survival in $A \subseteq X$}, starting from $x \in X$,
 if
 $ 
 {\mathbf{q}}(x,A)=\mathbf{q}(x,X) <1.
 $ 
\end{enumerate}
\end{defn}

According to the previous definition, when we talk about survival of the process (or extinction), we omit that this happens with positive probability
(resp. almost surely).
Results about survival and extinction in a set $A \subseteq X$ of the process has been studied in many papers: see for instance \cite{cf:BZ4, cf:BZ14-SLS,  cf:BZ2020, cf:Ligg1, cf:MachadoMenshikovPopov, cf:MenshikovVolkov, cf:PemStac1, cf:Z1} for discrete-time BRWs, \cite{cf:BZ2} for continuous-time BRWs and \cite{cf:BZ15, cf:GMPV09, cf:MP03} for BRWs in random environment.

While in homogeneous cases (i.e. breeding law which is independent of the site), through a coupling with a branching process, one can
easily establish where the process survives globally or not, in inhomogeneous cases the study is not trivial and in general it is still an open
question. A possible technique is to look for some regularity, like quasi-transitivity (see \cite[Section 2.4]{cf:BZ14-SLS} or \cite[Section 4.1]{cf:Z1}).
So far, the more general class of BRWs for which there is a simple characterization of global survival, is 
 the class of $\mathcal{F}$-BRWs (see for instance \cite{cf:BZ2, cf:BZ2017, cf:Z1} for the formal definition).
 This class includes homogeneous and nonhomogeneous quasi-transitive BRWs, but also more general inhomogeneous processes.
 Roughly speaking, an $\mathcal{F}$-BRW $(X, \boldmu)$ is a process which can be mapped onto a BRW $(U,\boldmu_1)$, where $U$ is finite.
 More precisely, there exists a map $g : X \mapsto U$ such that $G_\boldmu(\mathbf{z} \circ g)=G_{\boldmu_1}(\mathbf{z}) \circ g$ for every $\mathbf{z} \in[0,1]^U$. We refer to this map as a projection of $(X,\boldmu)$ onto $(U,\boldmu_1)$; in this case $\mathbf{q}(X)=\mathbf{q}(U)\circ g$, that is, $\mathbf{q}(X,x)=\mathbf{q}(U, g(x))$ for all $x \in X$. Therefore, either there is global survival for both processes or extinction for both.  This holds, even when $U$ is an infinite set. However when $U$ is finite, global survival for $(U, \boldmu_1)$ is equivalent to a simple inequality involving the Perron-Frobenius eigenvalue of the first-moment matrix of the process. 
 Indeed, let $M=(m_{x,y})_{x,y\in X}$ be the matrix whose entries are the expected number of offspring that an individual at $x$ places at $y$ during its lifetime, and let  $m_{x,y}^{(n)}$ be the entry of the $n$-power matrix $M^n$.
The $\mathcal{F}$-BRW survives globally if and only if $\liminf \sqrt[n]{\sum_{y \in X} m^{(n)}_{x,y}}>1$ (see \cite[Theorem 3.1]{cf:BZ14-SLS}).

Among BRWs with independent diffusion, there is a more explicit characterization of $\mathcal{F}$-BRWs which follows from the following proposition (see also \cite[Proposition 4.8]{cf:BCZ2024}): note that this is a result about projections where $U$ is not necessarily a finite set.
In this proposition, given a function $g:X \mapsto U$, we say that a family of offspring distributions $\{\rho_x\}_{x\in X}$ is $g$-invariant 
if $g(x)=g(y)$ implies 
$\rho_x =\rho_y$
for all $x,y \in X$. Moreover, we denote by $P$ the stochastic diffusion matrix as usual. 

\begin{pro}\label{prop:fbrwgen}
The following are equivalent.
\begin{enumerate}[(a)]
	\item
	There exists a surjective map $g:X\to U$
	such that, for every $g$-invariant family of offspring distributions $\{\rho_x\}_{x\in X}$,
	the resulting BRW with independent diffusion is projected on a BRW on $U$ with projection $g$.
	\item
	There exists a surjective map $g:X\to U$
	and a
	$g$-invariant family of offspring distributions $\{\rho_x\}_{x \in X}$ satisfying $\rho_x(0)<1$ for all $x \in U$, 
	such that 
	the resulting BRW with independent diffusion is projected on a BRW on $U$ with projection $g$.
	\item
	There exists a surjective  map $g:X\to U$ 
	such that 
	the quantity $\sum_{w\colon g(w)=u}p(x,w)$  only depends on $g(x)$ and $u$ (for all $x \in X$ and for all $u \in U$).
\end{enumerate}
In particular, if a BRW $(X, \boldmu)$ with independent diffusion on $X$ is projected on a BRW on $U$ with projection $g$ then the family $\{\rho^\boldmu_x\}_{x \in X}$ is $g$-invariant.

Suppose, in addition, that $\rho^\boldmu_x(0)<1$ for all $x \in X$; then the BRW with independent diffusion $(X, \boldmu)$ is projected on a BRW on $U$ with projection $g$ if and only if (c) holds and the family $\{\rho^\boldmu_x\}_{x \in X}$ is $g$-invariant.
\end{pro}
 \noindent Note that the last sentence of Proposition~\ref{prop:fbrwgen}, when $U$ is a finite set, characterizes $\mathcal{F}$-BRWs with independent diffusion.

A classical technique one can use to study survival and extinction of a BRW, is to couple it with another process whose behavior is known.
For example, if there is a stochastic coupling such that $(X,\boldmu)$ is dominated by a process which goes extinct, then the same can be said of
$(X,\boldmu)$. Even when a stochastic coupling cannot be found, germ order may hold.
Applications of germ order to BRWs has been extensively studied in \cite{cf:Hut2022} (for BRWs with independent diffusion where the offspring distribution does not depend on the position) and \cite{cf:BZgerm} (for general BRWs). While we refer the reader to these papers for details, we include here the main result on germ domination of \cite{cf:BZgerm}.

\begin{teo}\label{th:germorder}
Let $\boldmu \gegerm \boldnu$ (with $\delta<1$) and  $A \subseteq X$.
\begin{enumerate}
	\item If $x \in X$ then $\mathbf{q}^\boldmu(x,A) \le \mathbf{q}^\boldnu(x,A) (1-\delta)+\delta$.
 \item If $x \in X$, then $\mathbf{q}^\boldnu(x,A)<\mathbf{1}$ implies 
$\mathbf{q}^\boldmu(x,A)<\mathbf{1}$.
 \item If $\sup_{x \in X} \mathbf{q}^\boldnu(x,X)<1$, then $\mathbf{q}^\boldnu(x,A)=\mathbf{q}^\boldnu(x,X)$ 
for all $x \in X$ implies
 $\mathbf{q}^\boldmu(x,A)=\mathbf{q}^\boldmu(x,X)$ 
 for all $x \in X$.
\end{enumerate}
\end{teo}

The meaning of the previous theorem is that if $\boldmu \gegerm \boldnu$ then survival in $A$ for $(X,\boldnu)$ implies survival in $A$ for $(X,\boldmu)$.
Moreover, provided that $\sup_{x \in X} \mathbf{q}^\boldnu(x,X)<1$, strong
survival in $A$ for $(X,\boldnu)$ implies strong survival in $A$ for $(X,\boldmu)$.
Theorem~\ref{th:germorder} includes the case $\boldmu \gepgf \boldnu$, when $\delta=0$.
Note that $\boldmu \gegerm \boldnu$ does not imply that $\mathbf{q}^{\boldmu}(X) \le \mathbf{q}^{\boldnu}(X)$, but only that 
if $\mathbf{q}^{\boldnu}(X)<\mathbf1$, the same is true for $\mathbf{q}^{\boldmu}(X)$.
On the other hand, $\boldmu \gepgf \boldnu$ implies $\mathbf{q}^{\boldmu}(X) \le \mathbf{q}^{\boldnu}(X)$.

This allows us to state that, in general, independent diffusing BRWs do not have the largest nor the smallest global survival probability, when compared to
non-independently diffusing processes.
Indeed, if $(X,\boldnu)$ is an independently diffusing BRW, Example~\ref{exmp:nonind2} shows that a BRW that places all the offspring in one site
has a lower probability of global survival.
On the other hand, if $(X,\boldnu)$ is as in Example~\ref{exmp:nonind1}, a BRW $(X,\boldmu_1)$ which, in some sense, maximizes the occupancy of all sites,
has a greater probability of global survival.

Our aim is to show how we can use germ coupling to understand the behavior of BRWs, in particular by using comparisons with $\mathcal F$-BRWs since for these processes there is an easy characterization of global survival in terms of the first-moment matrix as discussed before Proposition~\ref{prop:fbrwgen}.

\begin{rem}\label{rem:germfbrw}
Let $(X, \boldnu)$ be an $\mathcal{F}$-BRW with independent diffusion, with a stochastic matrix $P$ and projection $g : X \mapsto U$, where $U$ is finite (see Proposition~\ref{prop:fbrwgen}). Consider the (finite) set of laws $\{\rho_x^\boldnu \colon x \in X\}$; since, by Proposition~\ref{prop:fbrwgen},  $g(x)=g(y)$ implies $\rho_x^\boldnu=\rho_y^\boldnu$, we use the notation $\rho_{g(x)}^\boldnu$ instead of $\rho_x^\boldnu$ for all $x \in X$. Roughly speaking, we are indexing the collection on $U$ instead of $X$.  
Given a family $\boldmu$ independently diffusing, with the same matrix $P$, it may happen that each $\rho_x^\boldmu \gegerm \rho^\boldnu_{g(x)}$ (or $\rho^\boldnu_{g(x)} \gegerm \rho_x^\boldmu$), with a $\delta_x=\delta_{g(x)}<1$ which depends only on $g(x)$, for all $x\in X$. 

Then, since $U$ is finite, $\sup_{x \in X} \delta_{g(x)} =\max_{u \in U} \delta_u <1$ whence, according to Theorem~\ref{pro:germindepdiff}, $\boldmu \gegerm \boldnu$ (resp.~$\boldnu \gegerm \boldmu$).
Clearly, if $\rho^\boldmu_x \gepgf \rho_{g(x)}^\boldnu$ (resp.~$\rho_{g(x)}^\boldnu \gepgf \rho^\boldmu_x$) for all 
$x \in X$
then $\boldmu \gepgf \boldnu$ 
(resp.~$\boldnu \gepgf \boldmu$). Note that, in this construction $(X,\boldmu)$ is not necessarily an $\mathcal{F}$-BRW.

A simple example is when $\rho_x^\boldnu=\widehat \rho$, for every $x \in X$. In this case $(X,\boldnu)$ is an $\mathcal{F}$-BRW (it can be projected on a branching process with reproduction law $\widehat\rho$, just take a singleton $U:=\{u\}$).
Suppose that $\rho^\boldmu_x \gegerm \widehat \rho$ (resp.~$\widehat \rho \gegerm \rho^\boldmu_x$) for every $x \in X$, with the same $\delta<1$.
Then $\boldmu\gegerm\boldnu$.
In particular,
if we denote by $\widehat \phi$ the generating function of $\widehat \rho$, according to Theorem~\ref{th:germorder}, if $\frac{\diff}{\diff t}\widehat \phi (1)>1$ (resp.~$\frac{\diff}{\diff t}\widehat \phi (1) \le 1$) then 
$(X,\boldmu)$ survives globally (resp.~dies out). 
 \end{rem}

As we already observed in the introduction, in general, the global behavior of a BRW is not determined by the first moments of the breeding laws:
it may even happen that all the first moments are larger than $1+\varepsilon$, and the BRW dies out (see \cite[Example 2.10]{cf:BRZ16}). 
On the other hand, for some processes, the knowledge of the first moment matrix is sufficient. For instance, let $(X,\boldmu)$ be
an independently diffusing
BRW, with $\rho^\boldmu_x\sim\mathcal{P}oi(\lambda_x)$, $\lambda_x\ge
1+\varepsilon$ for all $x\in X$ and for some $\varepsilon>0$.
Then $(X,\boldmu)$ stochastically dominates $(X,\boldnu)$, where $(X,\boldnu)$ is an independently diffusing BRW, with the same transition matrix $P$
and $\rho^\boldnu_x\sim\mathcal{P}oi(1+\varepsilon)$ for all $x\in X$, which survives globally.
Unfortunately, the usual stochastic order is not always an option: as we showed in Section~\ref{subsec:excompare}, as soon as
the total distributions are mixed Poisson, $\rho^\boldmu_x\sim\mathcal{P}oi(\Lambda_x)$, 
even if $\E[\Lambda_x]>1+\varepsilon$,
it is no longer true that one can use the classic stochastic order (see Examples~\ref{exmp:mix} and \ref{exmp:mix2}). In this case we need a more powerful tool as the germ order.

In the following examples, we consider offspring distributions associated to mixed Poisson laws.
In both examples, the idea is to compare a more general process with $(X,\boldnu)$, which is an independently diffusing $\mathcal{F}$-BRW with Poisson total distribution $\rho^\boldnu_x\sim
\mathcal{P}oi(\lambda_x)$ of which the mixed Poisson is a ``perturbation".
Note that if such a $(X,\boldnu)$ is an $\mathcal{F}$-BRW, then $\inf_{x \in X} \lambda_x =\min_{x \in X} \lambda_x >0$ and $\sup_{x \in X} \lambda_x =\max_{x \in X} \lambda_x < +\infty$. 
Recall that if $(X,\boldnu)$ is an $\mathcal{F}$-BRW, then it survives globally if $\liminf \sqrt[n]{\sum_{y \in X} m^{(n)}_{x,y}}>1$ and it goes globally extinct
if
$\liminf \sqrt[n]{\sum_{y \in X} m^{(n)}_{x,y}}\le1$ (where $m_{x,y}=\lambda_x p(x,y)$ for all $x,y\in X$).
Note that if the total distributions do not depend on $x$, that is $\rho^\boldnu_x\sim
\mathcal{P}oi(\lambda_x)$ for all $x\in X$, then $(X,\boldnu)$ is an $\mathcal{F}$-BRW.

 \begin{exmp}\label{ex:fbrw}
 Let $(X, \boldnu)$ be an independently diffusing $\mathcal{F}$-BRW with transition matrix $P$ and total distribution 
 $\{\rho_x^\boldnu\}_{x \in X}$ (where $\rho_x^\boldnu \sim \mathcal{P}oi(\lambda_x)$ where $\lambda_x>0$ for all $x \in X$).
      Given $\varepsilon \in (0, \min_{x \in X} \lambda_x)$ and $\{\alpha_x\}_{x \in X}$,  where $\alpha_x \in (0,1)$, let $(X,\boldmu)$ be the independently diffusing BRW, with the same transition matrix $P$ as $(X,\boldnu)$ and $\rho_x^\boldmu \sim \mathcal{P}oi(\Lambda_{x})$ where $\Lambda_{x} \sim \alpha_x \delta_{\lambda_x-\varepsilon}+(1-\alpha_x) \delta_{\lambda_x +\varepsilon}$. 
     If $\sup_{x \in X} \alpha_x <1/2$ and $(X,\boldnu)$ survives globally,
         then the $(X,\boldmu)$ survives globally; conversely,
    if $\inf_{x \in X} \alpha_x >1/2$
    and $(X,\boldnu)$ dies out,
     then the $(X,\boldmu)$ goes globally extinct
    (we use the fact that in the first case $\boldmu\gegerm\boldnu$, and in the second case $\boldnu\gegerm\boldmu$, see Example~\ref{exmp:mix2}).
 \end{exmp}

\begin{exmp}\label{ex:fbrw2}
Take a sequence $\{\widehat\lambda_x\}_{x\in X}$ such that
 $\sup_x \widehat\lambda_x<+\infty$ and
$\inf_x \widehat\lambda_x>\varepsilon\ge0$.
Let $(X,\boldmu)$ be the independently diffusing BRW with 
transition matrix $P$ and total distribution
$\rho^\boldmu_x\sim \mathcal{P}oi(\Lambda_{x})$, where $\Lambda_x\sim \alpha_x \delta_{\widehat\lambda_x-\varepsilon}+(1-\alpha_x) \delta_{\widehat\lambda_x +\varepsilon}$. 
Suppose that there exists an independently diffusing $\mathcal{F}$-BRW $(X,\boldnu)$ with the same transition matrix and total distribution
$\rho^\boldmu_x\sim\mathcal{P}oi(\lambda_{x})$, with $\widehat\lambda_x\ge\lambda^\boldnu_x$ 
(resp. $\widehat\lambda_x\le\lambda^\boldnu_x$), for all $x\in X$.
 Then if $(X,\boldnu)$ survives globally and
 $\sup_{x \in X} \alpha_x <1/2$,
then $(X,\boldmu)$ survives globally (resp. 
    if $\inf_{x \in X} \alpha_x >1/2$
    and $(X,\boldnu)$ dies out, then $(X,\boldmu)$ goes globally extinct).
   
    
    In particular, if $\lambda_{\min}:=\inf_x \widehat{\lambda}_x>1$ and $\sup_{x \in X} \alpha_x <1/2$, then $(X,\boldmu)$ 
    survives globally.  Analogously, if $\lambda_{\max}:=\sup_x\widehat\lambda_x\le1$ and $\inf_{x \in X} \alpha_x >1/2$, then $(X,\boldmu)$ 
    dies out. In fact, we compare with $(X,\boldnu)$ which has,  in the first case, $\lambda_x=\lambda_{\min}$ and in the second case
    $\lambda_x=\lambda_{\max}$, for all $x\in X$.
\end{exmp}

Note that in Examples~\ref{ex:fbrw}, \ref{ex:fbrw2},
with positive probability there are infinitely many particles breeding at rate smaller than 1 and nevertheless $(X,\boldmu)$ survives globally.

\section{Branching Processes in Varying Environment and germ order}
\label{sec:BPVE}

Consider a BRW $(\mathbb{N},\boldmu)$ with independent diffusion and transition matrix $P=\big (p(n,m) \big )_{n,m \in \mathbb{N}}$, where $p(n,m)$ equals $1$ if $m=n+1$ (for all $n \in \mathbb{N}$) and $0$ otherwise. Roughly speaking, each generation moves one step to the right,
so that $\mathbb{N}$ is interpreted as time instead of space. This BRW is called Branching Process in Varying Environment (shortly \textit{BPVE}), which is a branching process where the offspring distribution depends on time. The initial condition is usually $\eta_0=\delta_0$ (one particle at time 0).
For this process, we are interested only in survival in infinite subsets $A \subseteq \mathbb{N}$, since on the one hand it does not depend on the choice of the infinite set $A$ and, on the other hand, in finite sets there is no survival. For this reason, in this section we just say ``survival" without 
explicitly referring  to the (infinite) set. 

It is natural to compare these processes, with the ``fixed'' environment studied by Galton and Watson: in that case, provided that the probability of having one child is not 1, it is well known that there is extinction if and only if the expected number of children $m$ is smaller or equal to one. As soon as the reproduction law is not constant, survival does not depend only on the sequence of expected number of children $\{m_n\}_{n \in \N}$. On the one hand, there are examples of surviving BPVE, where $m_n<1$ for all $n\in \N$ (see Example~\ref{exmp:BPVE0}). On the other hand, given any sequence $\{m_n\}_{n \in \N}$ it is possible to find a BPVE which dies out and has  expected number of children $\{m_n\}_{n \in \N}$ (see  \cite[Example 2.10]{cf:BRZ16}).

Let us denote by $\phi^\boldmu_n(t):=\sum_{i \in \mathbb{N}} \rho^\boldmu_n(i) t^i$ the generating function of the total offspring distribution at time $n$.
To avoid a trivial situation, we suppose henceforth that $\rho^\boldmu_n(0)<1$, for all $n \in \mathbb{N}$
(otherwise the process, starting with $\eta_0=\delta_0$, would die out almost surely). Under this assumption, the BPVE survives, starting with $\eta_0=\delta_0$,  if and only if it survives starting with $\eta_0=\delta_n$,  (for any fixed $n \in \mathbb{N}$): indeed, while the ``if" part is trivial, the ``only if" part follows from the fact that there is a positive probability that the progeny of a particle at time $0$ survives up to time $n$. Moreover, consider a strictly increasing sequence of integers $\{n_i\}_{i \in \mathbb{N}}$; a BPVE $\{\eta_n\}_{n \in \mathbb{N}}$ survives if and only if there are particles alive at time $n_i$ for all $i \in \mathbb{N}$, that is, if and only if $\{\eta_{n_i}\}_{i \in \mathbb{N}}$ survives. Clearly, the generating functions of the process $\{\eta_{n_i}\}_{i \in \mathbb{N}}$ are $\{\varphi_i\}_{i \in \mathbb{N}}$ where
\begin{equation}
    \label{eq:blocks}
    \varphi_i:=\phi_{n_{i+1}-1} \circ \ldots \circ \phi_{n_i}, \quad \forall i \in \mathbb{N}
\end{equation}
and $\{\phi_n\}_{n \in \mathbb{N}}$ are the generating function of the process $\{\eta_n\}_{n \in \mathbb{N}}$.
This remark along with Theorem~\ref{th:germorder} yields the following result,
which allows to study the process  observing it only along a subsequence of times.

\begin{pro}\label{pro:germBPVE}
Let us consider two BPVEs $(\mathbb{N},\boldmu)$ and $(\mathbb{N},\boldnu)$    
with total offspring generating functions $\{\phi^\boldmu_n\}_{n \in \mathbb{N}}$ and
$\{\phi^\boldnu_n\}_{n \in \mathbb{N}}$ respectively. Let 
$\{\varphi^\boldmu_n\}_{n \in \mathbb{N}}$ and
$\{\varphi^\boldnu_n\}_{n \in \mathbb{N}}$ be as in Equation~\eqref{eq:blocks},  namely
\[
\begin{cases}
    \varphi^\boldmu_i:=\phi^\boldmu_{n_{i+1}-1} \circ \ldots \circ \phi^\boldmu_{n_i},\\
        \varphi^\boldnu_i:=\phi^\boldnu_{m_{i+1}-1} \circ \ldots \circ \phi^\boldnu_{m_i},\\
\end{cases}
  \quad \forall i \in \mathbb{N},
\]
where 
$\{n_i\}_{i \in \mathbb{N}}$ and $\{m_i\}_{i \in \mathbb{N}}$ are fixed (strictly) increasing sequences of integers.
Suppose that there exists $\delta \in [0,1)$ such that $\varphi^\boldmu_n(t) \le \varphi^\boldnu_n(t)$ for all $t \in [\delta,1]$ and $n \in \mathbb{N}$. Then if $(\mathbb{N},\boldnu)$ survives, so does $(\mathbb{N},\boldmu)$.
\end{pro}

Note that in the previous proposition, only the generating functions 
$\{\phi^\boldmu_n\}_{n \ge n_0}$ and
$\{\phi^\boldnu_m\}_{m \ge m_0}$ are involved. In particular, if there exists $\delta \in [0,1)$ such that $\phi^\boldmu_n(t) \le \phi^\boldnu_n(t)$ for all $t \in [\delta,1]$ and $n \ge n_0$ (for some $n_0 \in \mathbb{N}$), then the proposition applies by taking $\{n_i\}_{i \in \mathbb{N}}$ and $\{m_i\}_{i \in \mathbb{N}}$ as $n_i=m_i:=n_0+i$ for all $i \in \mathbb{N}$.

The main use for Proposition~\ref{pro:germBPVE} is to find sufficient conditions for survival or extinction for BPVEs. Fix a suitable surviving process $\boldnu$ (resp.~non-surviving process $\boldmu$): by handling the inequality $\phi^\boldmu_n(t) \le \phi^\boldnu_n(t)$ one could find such conditions. A full and deep analysis of this topics goes beyond the purpose of this paper, nevertheless we can describe some simple conditions of this kind (see Example~\ref{exmp:BPVE1} and Proposition~\ref{pro:BPVE1}).
To this aim, note that, for all $t \in [0,1]$,
\[
\phi^\boldmu_n(t)=
1+(t-1)\sum_{i \in \mathbb{N}} \Big (\sum_{j \ge i+1} \rho^\boldmu_n(j) \Big ) t^i.
\]
This leads to the following lemma.

\begin{lem}\label{lem:donotwriteanotherlemmaplease}
For every $t \in [0,1]$, $n \in \mathbb{N}$, the inequality $\phi^\boldmu_n(t) \le \phi^\boldnu_n(t)$ is equivalent to
\begin{equation}\label{eq:reverseinequality1}
\sum_{i \in \mathbb{N}} \Big (\sum_{j \ge i+1} \rho^\boldmu_n(j) \Big ) t^i \ge \sum_{i \in \mathbb{N}} \Big (\sum_{j \ge i+1} \rho^\boldnu_n(j) \Big ) t^i.
\end{equation}
Moreover, if the following inequality holds for some $k_0\in\N$
\begin{equation}\label{eq:reverseinequality2}
\sum_{i=0}^{k_0} \Big (\sum_{j \ge i+1} \rho^\boldmu_n(j) \Big ) t^i \ge \sum_{i \in \mathbb{N}} \Big (\sum_{j \ge i+1} \rho^\boldnu_n(j) \Big ) t^i,
\end{equation}
then it holds for all $k\ge k_0$ and \eqref{eq:reverseinequality1} holds (thus also $\phi^\boldmu_n(t) \le \phi^\boldnu_n(t)$).
\end{lem}

We give an example of BPVE where at each time the support of the breeding law is $\{0,1\}$ and yet the process survives.
\begin{exmp}\label{exmp:BPVE0}
Let $(\N,\boldnu)$ be a BPVE, with $\rho^\boldnu_n(1)=a_n$ and $\rho^\boldnu_n(0)=1-a_n$, where $a_n\in(0,1]$, for all $n\in \N$.
The process survives if and only if $\sum_{n \in \mathbb{N}} (1-a_n)<+\infty$ (the probability of survival being $\prod_{n \in \mathbb{N}} a_n>0$). 
\end{exmp}

\begin{exmp}\label{exmp:BPVE1}
Let $(\N,\boldmu)$ be a BPVE, with $\limsup_{n \to +\infty} (2\rho_n^\boldmu (0)+\rho_n^\boldmu(1))<\delta < 1$, then 
$(\N,\boldmu)$ survives.
\\
We apply \eqref{eq:reverseinequality2}, with $k_0=1$, to compare $(\N, \boldmu)$ with the process $(\N, \boldnu)$ of Example~\ref{exmp:BPVE0}.
We have that $\sum_{i \in \mathbb{N}} (\sum_{j \ge i+1} \rho^\boldnu_n(j) ) t^i=a_n$
and $\sum_{i=0}^{1}  (\sum_{j \ge i+1} \rho^\boldmu_n(j)  ) t^i=1-\rho_n^\boldmu(0)+(1-\rho_n^\boldmu(0)-\rho_n^\boldmu(1))t$.
Now take take $\delta$ such that $\limsup_{n \to +\infty} (2\rho_n^\boldmu (0)+\rho_n^\boldmu(1))<\delta < 1$. For all sufficiently large $n$,
\[
\begin{split}
1-\rho_n^\boldmu(0)+(1-\rho_n^\boldmu(0)-\rho_n^\boldmu(1))\delta 
& \ge 1-(\delta -\rho_n^\boldmu(0)-\rho_n^\boldmu(1))+(1-\rho_n^\boldmu(0)-\rho_n^\boldmu(1)) \delta\\
& = 1+ (1-\delta)(\rho_n^\boldmu(0)+\rho_n^\boldmu(1)) \ge 1 \ge a_n.
\end{split}
\]
The proof is complete once we note that $1-\rho_n^\boldmu(0)+(1-\rho_n^\boldmu(0)-\rho_n^\boldmu(1))t\ge 1-\rho_n^\boldmu(0)+(1-\rho_n^\boldmu(0)-\rho_n^\boldmu(1))\delta$ for all $t\in[\delta,1]$.
\end{exmp}

The argument of the previous example can be extended, with a similar argument, to the statement of the following proposition.

\begin{pro}
    \label{pro:BPVE1}
    If  there exists $k_0 \in \mathbb{N}\setminus\{0\}$ such that $\limsup_{n \to +\infty} \sum_{k=0}^{k_0} (k_0+1-k) \rho^\boldmu_n(k) <k_0$ then the BPVE $(\mathbb{N}, \boldmu)$ survives.
\end{pro}

Note that the hypothesis of Proposition~\ref{pro:BPVE1}, namely $\limsup_{n \to +\infty} \sum_{k=0}^{k_0} (k_0+1-k) \rho^\boldmu_n(k) <k_0$, is implied by (but does not imply) $\liminf_{n \to +\infty} \sum_{k=0}^{k_0} k \rho_n^\boldmu(k) >1$. Nevertheless a partial converse is true: if $\limsup_{n \to +\infty} \sum_{k=0}^{k_0} (k_0+1-k) \rho^\boldmu_n(k) <k_0$, then $\liminf_{n \to +\infty} \sum_{k \in \mathbb{N}} k \rho_n^\boldmu(k) >1$
(the reverse implication does not hold).

Besides giving new conditions for survival or extinction of a BPVE, germ order allows to relax some conditions which are known in literature. Again, we are not planning to go into details but here is an example, which applies and extends \cite[Theorem 2.6]{cf:BRZ16}.

\begin{teo}
    \label{pro:bertrodzuc}
    Let us consider two BPVEs $(\mathbb{N},\boldmu)$ and $(\mathbb{N},\boldnu)$    
with total offspring generating functions $\{\phi^\boldmu_n\}_{n \in \mathbb{N}}$ and
$\{\phi^\boldnu_n\}_{n \in \mathbb{N}}$ respectively. Suppose that there exists $\delta \in [0,1)$ and $n_0 \in \mathbb{N}$ such that $\phi^\boldmu_n(t) \le \phi^\boldnu_n(t)$ for all $t \in [\delta,1]$ and $n \ge n_0$. Define the $k$th moment of the law $\rho_n^\boldnu$ as $m_n^{(k)}:=\sum_{i \in \mathbb{N}} i^k p_n^\boldnu(i)$.
Suppose that $m^{(2)}_n<+\infty$ for every sufficiently large $n\in \mathbb{N}$. If 
\begin{equation}
    \label{eq:bertrodzuc1}
    \begin{cases}
        \sum_{j =n_0}^\infty \frac{m_j^{(2)}-m^{(1)}_j}{m^{(1)}_j} \Big ( \prod_{i=n_0}^j m^{(1)}_i  \Big )^{-1}<+\infty \\
        \inf_{j \in \mathbb{N}} \prod_{i=n_0}^j m^{(1)}_i >0
    \end{cases}
\end{equation}
then $(\mathbb{N}, \boldmu)$ survives.
\end{teo}

The proof is straightforward: if Equation~\eqref{eq:bertrodzuc1} holds, then according to \cite[Theorem 2.6]{cf:BRZ16}, $(\mathbb{N}, \boldnu)$ survives, thus by Proposition~\ref{pro:germBPVE} $(\mathbb{N}, \boldmu)$ survives. The point here is that the moments of $(\mathbb{N}, \boldmu)$, unlike those of $(\mathbb{N}, \boldnu)$, do not need to be finite. The same approach can be used to extend other results in the theory of branching processes in varying environment; to this aim, see for instance \cite{cf:Agresti75, cf:BM08, cf:Bigg, cf:Faern}.

\section{Rumor Processes}
\label{sec:firework}

The main purpose of this section is to generalize the rumor processes studied in \cite{cf:BZ13}: the firework and the reverse firework processes.
Here $Y$ represents the set of vertices of a non-oriented graph (or a discrete metric space); on the other hand, $X$ represents the type of fireworks that we can place at each point of $X$.
In this interpretation, $\mu_y$ is the law of the random choice of stations/fireworks placed at $y$: a function $f \in S_X$ is chosen according to $\mu_y$ and $f(x)$ stations/fireworks of type $x$ are placed at $y$ (for every $x \in X$). Note that, while the configurations at different points $y \neq \hat y$ are independent, in a fixed point $y$ the number and types of stations/fireworks are not necessarily independent. More precisely, they are independent if and only if 
$\boldmu$ is an independently diffusing family. 
For each type and position we have a (possibly different) distribution for the radius, say $\mathfrak{R}:=\{\mathfrak{r}_{x,y}\}_{x \in X, y \in Y}$, where $\mathfrak{r}_{x,y}$ is a probability measure defined on $[0, +\infty]$.
The firework and the reverse firework processes are identified by the choice of $(\boldmu,\mathfrak{R})$.

The explicit construction can be carried out by choosing a family $\{F_y, \, \{R_{y,x,i}\}_{x \in X, i \in \mathbb{N}}\}_{y \in Y}$ of independent random variables such that $F_y \sim \mu_y$ for all $y \in Y$ and $R_{y,x,i} \sim \mathfrak{r}_{x,y}$ for all $y \in Y$, $x \in X$ and $i \in \mathbb{N}$ (when we need the explicit value of $F_y$ we write $F_y(\omega,x)$ where $\omega \in \Omega$ and $x \in X$); in particular, we denote by $F_y(x)$ the random variable $\omega \mapsto F_y(\omega, x)$ and by $F_y(\omega)$ the function
$x \mapsto F_y(\omega,x)$. The interpretation is as follows: a random number $F_y(\omega,x)$ of stations of type $x$ are placed at $y$ and $R_{y,x,i}(\omega)$ is the radius of the $i$-th station of type $x$ at $y$. 
We denote by $\widetilde R_y:=\max\{\ident_{\{F_y(x)>0\}} \cdot R_{y,x,i} \colon x \in X, i=1, \ldots, F_y(x)\}$ the maximum radius at $y$. It is trivial to see that, for every $y \in Y$, $t \ge 0$,
\begin{equation}\label{eq:radiusmaxlaw}
\pr(\widetilde R_y < t)=\sum_{f \in S_X} \mu_y(f) \prod_{x \in X} \mathfrak{r}_{x,y}([0, t))^{f(x)}= G_\boldmu(\mathbf{\mathfrak{r}}_y(t)|y),
\end{equation}
where $\mathbf{\mathfrak{r}}_y(t) \in [0,1]^X$ is defined by $\mathbf{\mathfrak{r}}_y(t,x):=\mathfrak{r}_{x,y}([0,t))$. We assume that $ \mathfrak{r}_{x,y}([0,1))\in(0,1)$ for all $x \in X$.

Let $o \in Y$ be a fixed point that we call \textit{root} and let $(Y,E_Y)$ be a graph endowed with the usual distance.
We describe the firework process starting from $o$,
with deterministic  number of stations of each type and radii.
If either the number of stations
or the radii are random (or the graph itself is a random graph),
then this evolution applies to every fixed configuration.
At time $0$ only the stations at $o$, if any, are active. At time $1$ all the stations at $y \in Y$ such that
the distance between $o$ and $y$ is less than or equal to $\widetilde R_o$ are activated.
Given the set $A_n$ of active stations at time $n$, we have that
$y \in A_{n+1}$ if and only if (1) there exists $w \in  A_n$ such that
the distance between $y$ and $x$ is less than or equal to $\widetilde R_w$ and (2) $y \not \in \bigcup_{i=0}^n A_i$ (that is, the stations at $y$ are still inactive).
Clearly if there are no stations at $o$ the process does not start at all.
We say that the process survives if and only if 
$A_n\neq\emptyset$ for all $n\in\N$.
    \begin{defn}
        Given a firework process $(\boldmu,\mathfrak{R})$, on a random graph $(Y,E_Y)$,
        we say that there is survival 
if the event
\[
 V=\{\omega \in \Omega\colon A_n(\omega)\neq\emptyset, \text{ for all }n\in\N\}
\]
has positive probability.
    \end{defn}
    We call this \textit{annealed survival} (see Section~\ref{subsec:annealed}
for the definition of the single-station counterpart of the process).
We are mainly interested in the \textit{quenched survival}, that is, the (positive) probability
of survival conditioned on the number and types of stations (see Section~\ref{subsec:N} for details).

The reverse firework process, starting from the root $o \in Y$, evolves slightly differently. 
While in the firework process the stations are actively
sending and passively listening, in the reverse firework process it is the other way around.
At time 0 all stations at $o$ are active. Their radii do not play a role in the evolution.
At time $1$ all stations at $y$ such that
the distance between $o$ and $y$ is less than or equal to $\widetilde R_y$ are activated.
Given the set $B_n$ of active stations at time $n$, we have that
$y \in B_{n+1}$ if and only if there exists $w \in  B_n$ such that (1)
the distance between $w$ and $y$ is less than or equal to $\widetilde R_y$ and (2) $y \not \in \bigcup_{i=0}^n B_i$.
The definition of annealed survival of the reverse firework is as follows (quenched survival is analogous, conditioning on given configurations of stations and radii).
 \begin{defn}
        Given a reverse firework process $(\boldmu,\mathfrak{R})$, on a random graph $(Y,E_Y)$,
        we say that there is survival 
if the event
\[
 S=\{\omega \in \Omega\colon B_n(\omega)\neq\emptyset, \text{ for all }n\in\N\}
\]
has positive probability.
    \end{defn}

Henceforth, we denote a rumor process (firework or reverse firework) by $(\boldmu,\mathfrak{R})$ to emphasize the dependence on the family of station laws $\boldmu$ and the family of radius laws $\mathfrak{R}$.

\begin{rem}
    \label{rem:novelty}
By construction, the behavior of a firework or reverse firework process depends exclusively on the family of maximal radii $\{\widetilde R_y\}_{y \in Y}$; this means that if two such processes have the same family of maximal radii, then they have the same behavior. It is natural to ask where is the novelty of a multi-type station approach compared to the single-type station of previous papers.

To answer this question, let us see first when two rumor processes are equivalent.
 Observe that the types $x_1\neq x_2$ are actually different only if $\mathfrak{r}_{x_1,y}\neq \mathfrak{r}_{x_2,y}$ for some $y \in Y$. Indeed, when $\mathfrak{r}_{x_1,y}=\mathfrak{r}_{x_2,y}$ for all $y \in Y$ there is an equivalent process obtained by identifying $x_1$ and $x_2$ as a single point $\overline x \not \in X$; thus, the new set of types is $(X\cup \{\overline x\})\setminus\{x_1,x_2\}$ with radius $\mathfrak{r}_{\overline x,y}=\mathfrak{r}_{x_1,y}$ for all $y \in Y$. When in the original process a station of type $x_1$ or $x_2$ is created at $y\in Y$, in the new process a station of type $\overline{x}$ is generated at the same $y$.

In addition, we observe that if $\boldmu$ is an independently diffusing family, with total distributions $\{\rho_y\}_{y\in Y}$ and matrix $P$, there is an equivalent process (with the same maximum radius law) obtained by generating a number of single-type stations according to $\rho_y$ with radius law
$\sum_{x \in X} p(y,x)\mathfrak{r}_{x,y}$. In fact, using Equation~\eqref{eq:Gindepdiff}, we have that Equation~\eqref{eq:radiusmaxlaw} becomes $\pr(\widetilde R_y < t)=\rho_y \big (\sum_{x \in X} p(y,x) \mathfrak{r}_{x,y}([0,t))  \big )$.

This equivalence fails when $\boldmu$ is not an independently diffusing family and there exists $x_1 \neq x_2$ such that $\mathfrak{r}_{x_1,y}\neq \mathfrak{r}_{x_2,y}$ (for some $y \in Y$): in this case the process is not reducible to the single-type station frameset of previous works. Indeed, in the general case, even though $\{F_y\}_{y\in Y}$ are independent, the random numbers of stations of each type at a fixed site are not independent.
\end{rem}

\subsection{The process on $\N$}
\label{subsec:N}

From now on we pick $Y=\N$ and $o=0$.
We call the process \textit{homogeneous} if $\mu_y=\mu$ does not depend on $y \in \N $
and $\mathfrak{r}_{x,y}=\mathfrak{r}_x$ does not depend on $y \in \N$; we call the process \textit{heterogeneous} otherwise. In the case of a reverse firework, we allow the root to have a (possibly) different law $\mu_0$ event when the process is homogeneous. 
The construction of rumor processes on $\N$
is a particular case of the construction of a labeled
Galton-Watson tree given in \cite{cf:BZ13}, where only one type of station was considered. The same construction could be carried out here for our multi-type station approach. 
We want to stress
that all the results of \cite{cf:BZ13}, including those on the Galton-Watson trees, can be extended to the multi-type station case as well.

Given $\boldmu=\{\mu_i\}_{i \in \mathbb{N}}$
and the collection $\{F_i\}_{i \in \N}$ of independent $S_X$-valued random variables (where $F_i \sim \mu_i$), we define the environment as the 
$(S_X)^\N$-valued random variable
$\mathcal{N}$ where $\mathcal{N}(\omega):=\{F_i(\omega)\}_{i \in \N}$.
The canonical probability space is $((S_X)^{\mathbb{N}}, \theta)$
where $\theta=\prod_{i \in \N}\mu_i$ is the law of $\mathcal{N}$.

Let $\upsilon= \prod_{n \in \mathbb{N}, x \in X, i \in \N^*} \pr_{R_{n,x,i}}$ be the product measure of the
laws of $\{R_{n,x,i}\}_{n \in \mathbb{N}, x \in X, i \in \N^*}$ on the canonical product space $\mathcal{O}:=[0,+\infty)^{\mathbb{N}\times X \times \N^*}$ (where $\N^*:=\N \setminus \{0\}$).
Henceforth we do not need the original probability space $(\Omega, \mathcal{F}, \pr)$ and we consider
$\Omega:= (S_X)^\N \times \mathcal{O}$ endowed with the usual product $\sigma$-algebra and
$\pr:= \theta \times \upsilon$. With a slight abuse of notation we denote again
by $\mathcal{N}$, $F_n$ and $R_{n,x,i}$ the natural counterparts of the original random variables which are now defined on the ``new''
space $\Omega$. For every $\omega \in \Omega$ the processes evolve according to the sequence of radii
$\big \{ 
R_{n,x,i}(\omega) 
\big \}_{n\in \N, x \in X, i \in \{1,\ldots,F_n(\omega,x) \}}$
(see also Section~\ref{subsec:annealed}).
Extinction and survival are 
measurable sets with respect to the product $\sigma$-algebra. By standard measure theory, for every event $A$ we have
\begin{equation}\label{eq:disintegration}
	\pr(A)=\int_{(S_X)^\N} \pr(A|\mathcal{N}=\mathbf{f})\theta(\diff \mathbf{f})
\end{equation}
where
$\pr(A | \mathcal{N}=\mathbf{f})=\upsilon( \mathbf{r} \in \mathcal{O} \colon (\mathbf{f},\mathbf{r}) \in A)$
since $\mathcal{N}(\mathbf{f}, \mathbf{r})\equiv\mathbf{f}$.
We call $\pr(A)$ \textit{annealed} probability of $A$ and $\pr(A | \mathcal{N}=\mathbf{f})$ \textit{quenched} probabilities of $A$.
Using Equation~\eqref{eq:disintegration}, we have that
$\pr(A)=0$ (resp.~$\pr(A)=1$) if and only if $\pr(A | \mathcal{N}=\mathbf{f})=0$ (resp.~$\pr(A | \mathcal{N}=\mathbf{f})=1$)
$\theta$-almost surely.
It is clear that $\pr(A)>0$ if and only
if $\theta(\mathbf{f}\colon \pr(A| \mathcal{N}=\mathbf{f})>0)>0$.
More precise results are proven in Section~\ref{subsec:mainrumor}.

\subsection{The single-station counterpart}
\label{subsec:annealed}

Here, in order to construct firework processes, we take multiple randomized steps. First we randomly choose an environment $\mathbf{f} \in (S_X)^\N$ according to the law $\theta$; essentially, we choose a random number/types of stations). Once the environment is chosen, the firework and the reverse firework are random processes where the stochastic nature comes from the random choice of the radii of the stations that we make according to the law $\upsilon$.
We associate to
our (firework or reverse firework) process
with random numbers of stations, a (firework or reverse firework)  process with one station per vertex and a radius which is the maximum among the radii of all the (random) stations placed at each site in the original process.
To be precise, let us consider the process with one station on each vertex $y$
and radius $\widetilde R_y=  \max_{x \in X, i \in \N^*} \big ( \ident_{\{F_y(x) >0\}} \cdot R_{y,x,i} \ident_{[1,F_y(x)]}(i) \big )$ at $y \in \N$.
We call this process the  
\textit{single-station counterpart}
of the original process.
The single-station counterpart does not retain any information about the environment, nevertheless its 
probability of survival is the same as the annealed probability of survival of the original process (as defined by Equation~\eqref{eq:disintegration}.
In the following we use extensively
the 
fact that 
$\pr(\widetilde R_y<t) =
\ident_{[0,+\infty)}(t) G_\boldmu(\bm{\mathfrak{r}}_y(t)|y)
$ (as seen in Equation~\eqref{eq:radiusmaxlaw}).
Since we assumed that $ \mathfrak{r}_{x,y}(1)\in(0,1)$, then $\pr(\widetilde R_{y} < 1)\in(0,1)$.
The strategy is to draw information about the quenched probabilities $\pr(A | \mathcal{N}=\mathbf{f})$ from the annealed probability $\pr(A)$, which are linked by Equation~\eqref{eq:disintegration}.

It is clear, from Equation~\eqref{eq:disintegration}, that $\pr(A)=0$ if and only if $\pr(A|\mathcal{N}=\mathbf{f})=0$ for $\theta$-almost every $\mathbf{f} \in (S_X)^\N$.
Conversely, since there is no survival if there are no stations at the origin, then
the probability that the environment can sustain a surviving process cannot exceed the probability of
the event ``there are stations at the origin'', that is,
$\pr({F_0}\neq \mathbf{0})=1-\mu_0(\mathbf{0})$. The results in Section~\ref{subsec:mainrumor} tell us that
when the probability that the environment can sustain a surviving process is greater than $0$ then it
reaches the maximum admissible value $1-\mu_0(\mathbf{0})$.

\subsection{Results for multi-type firework and reverse firework}
\label{subsec:mainrumor}

In this section we study conditions for survival/extinction of rumor processes,
which extend the results for the model in \cite{cf:BZ13} to the multi-type stations version of these processes.
Recall that for the firework process, $V$ is the event in which all vertices are reached by a signal, while for the reverse
firework, the same event is named $S$.

\begin{teo}\label{thm:2} \textbf{(Homogeneous Firework on $\mathbb{N}$).}
 Consider the homogeneous case $\mu_n=\mu$ and $\mathfrak{r}_{x,n}=\mathfrak{r}_x$ for all $n \in \mathbb{N}$, $x \in X$, that is, 
 $G_{\boldmu}(\mathbf{\mathfrak{r}}_n(t)|n)=G_{\boldmu}(\mathbf{\mathfrak{r}}(t)|0)$ 
 for all $n \in \mathbb{N}$, where $\mathbf{\mathfrak{r}}(t,x)=\mathfrak{r}_x([0,t)])$.
\begin{enumerate}
\item
If
\begin{equation} \label{eq:sumprod2.5}
 \sum_{n=0}^\infty \prod_{i=0}^n G_{\boldmu}(\bm{\mathfrak{r}}(i+1)|0)=+\infty,
\end{equation}
then 
$\pr(V)=0$ and 
$\theta(\mathbf{f}\colon \pr(V|\mathcal{N}=\mathbf{f})=0)=1$;
\item if
\begin{equation} \label{eq:sumprod2}
\sum_{n=0}^\infty \prod_{i=0}^n G_{\boldmu}(\bm{\mathfrak{r}}(i+1)|0)<+\infty,
\end{equation}
then $\pr(V)>0$ and
$\theta(\mathbf{f}\colon \pr(V|\mathcal{N}=\mathbf{f})>0)=1-\mu(\mathbf{0})$.
\end{enumerate}
\end{teo}

Observe that, since Equation~\eqref{eq:sumprod2.5} holds if and only if Equation~\eqref{eq:sumprod2} does not, Theorem~\ref{thm:2} 
characterizes extinction and survival and states that annealed extinction is equivalent to quenched extinction
for almost any configuration of the stations, while annealed survival is equivalent to quenched survival for almost any configuration with at least one station at 0. 
 More precisely
\begin{equation}\label{eq:betterthm2}
\begin{array}{ccccc}
 \eqref{eq:sumprod2.5}
&\Longleftrightarrow & \theta(\mathbf{f}\colon \pr(V|\mathcal{N}=\mathbf{f})=0)=1  &\Longleftrightarrow & \pr(V)=0 \\
\eqref{eq:sumprod2}
&\Longleftrightarrow & \theta(\mathbf{f}\colon \pr(V|\mathcal{N}=\mathbf{f})>0)=1-\mu(\mathbf{0}) &\Longleftrightarrow &\pr(V)>0,
\end{array}
\end{equation}
(see the proof of Theorem~\ref{thm:2} for details).

 Recall that in the reverse firework the number of stations at the origin does not matter as long as it is strictly positive; if there are no stations at the origin then the process does not start. For this reason, even whe the reverse firework process is homogeneous, $\mu_0$ might be different from $\mu_n$ (where $n \in \N \setminus\{0\}$).
We have the following result.
\begin{teo}\label{th:reversehomogeneous} \textbf{(Homogeneous Reverse Firework on $\mathbb{N}$).}
 Consider the homogeneous case $\mu_n=\mu$ for all $n \in \N\setminus\{0\}$ and $\mathfrak{r}_{x,n}=\mathfrak{r}_x$ for all $n \in \mathbb{N}$, $x \in X$, that is, 
 $G_{\boldmu}(\mathbf{\mathfrak{r}}_n(t)|n)=G_{\boldmu}(\mathbf{\mathfrak{r}}(t)|1)$ 
 for all $n \in \mathbb{N}\setminus\{0\}$, where $\mathbf{\mathfrak{r}}(t,x)=\mathfrak{r}_x([0,t))$.
  Define
$W:=\sum_{i \in \N} (1-G_\boldmu(\mathfrak{r}(i)|1))$ (or $W:=\int_0^\infty (1-G_\boldmu(\mathfrak{r}(t)|1)) \diff t$).
\begin{enumerate}
 \item If $W=+\infty$, then $\theta(\mathbf{f}\colon \pr(S|\mathcal{N}=\mathbf{f})=1)=1-\mu_0(\mathbf{0})$. 
\item If $W<+\infty$, then $\theta(\mathbf{f}\colon \pr(S|\mathcal{N}=\mathbf{f})=0)=1$.
\end{enumerate}
\end{teo}

Theorem~\ref{thm:3} characterizes annealed and quenched extinction/survival of the homogeneous reverse firework process.
More precisely
\begin{equation}\label{eq:betterthm3}
\begin{array}{ccccc}
    W=+\infty  &\Longleftrightarrow &\theta(\mathbf{f}\colon \pr(S|\mathcal{N}=\mathbf{f})=1)=1-\mu_0(\mathbf{0}) &\Longleftrightarrow& \pr(S)>0,\\
     W<+\infty  &\Longleftrightarrow &\theta(\mathbf{f}\colon \pr(S|\mathcal{N}=\mathbf{f})=0)=1 &\Longleftrightarrow& \pr(S)=0.
\end{array}
\end{equation}
We emphasize that if $\E[\widetilde R_1]<+\infty$ then both homogeneous rumor processes die out; for the reverse firework, in particular the condition $\E[\widetilde R_1]<+\infty$ is equivalent to extinction.

In the heterogeneous case, for the firework process on $\N$, we have a sufficient condition for survival.

\begin{teo}\label{thm:3} \textbf{(Heterogeneous Firework on $\mathbb{N}$)}.
 In the heterogeneous case, if
\begin{equation}\label{eq:sumprod3}
 \sum_{n=0}^\infty \prod_{i=0}^n G_{\boldmu}(\mathbf{\mathfrak{r}}_{i} (n-i+1)|i))<+\infty,
\end{equation}
then $\pr(V)>0$ and $\theta(\mathbf{f}\colon\pr(V|\mathcal{N}=\mathbf{f})>0)=1-\mu_0(\mathbf{0})$. 
\end{teo}

For the reverse firework on $\N$, we have the following result.

\begin{teo}\label{th:reverseinhomogeneous}\textbf{(Heterogeneous Reverse Firework on $\mathbb{N}$)}.
Consider the heterogeneous reverse firework process on $\N$.
\begin{enumerate}
 \item
$\sum_{k \ge 1} (1-G_{\boldmu}(\mathbf{\mathfrak{r}}_{n+k}(k)|n+k))=+\infty$ for all $n \in \N$ if and only if
$\theta(\mathbf{f}\colon \pr(S|\mathcal{N}=\mathbf{f})=1)=
1-\mu_0(\mathbf{0})
$.
\item
If $\sum_{n \in \N} \prod_{k=1}^\infty G_{\boldmu}(\mathbf{\mathfrak{r}}_{n+k}(k)|n+k)<+\infty$,
then $\theta(\mathbf{f}\colon \pr(S|\mathcal{N}=\mathbf{f})>0)=1-\mu_0(\mathbf{0})$.
\end{enumerate}
\end{teo}

\begin{rem}\label{rem:kummer}
In the previous theorems, it is relevant to study the convergence (or divergence) of positive series. Equivalent conditions can be found by using Kummer's convergence test (see Lemma~\ref{lem:kummer}). Let us consider the conditions in Theorem~\ref{thm:2} for instance: Equation~\eqref{eq:sumprod2} holds if and only if $G_{\boldmu}(\bm{\mathfrak{r}}(n)|0)=B_n/(\alpha+B_{n+1})$ for all $n \ge N$, for some $\alpha>0$, $N \ge 0$ and a sequence $\{B_n\}_{n \ge N}$ of strictly positive numbers. Equivalently, Equation~\eqref{eq:sumprod2} holds if and only there exists $\alpha>0$, $N\ge 0$ and $B_N>0$, such that $B_{n+1}:= B_n/G_{\boldmu}(\bm{\mathfrak{r}}(n)|0)-\alpha$ defines recursively a strictly positive sequence.
\end{rem}

\subsection{Germ order applied to firework and reverse firework processes}\label{subsec:firework-germ}

We would like to compare the behavior of two rumor processes, 
$(\boldmu,\mathfrak{R})$ and $(\boldnu,\bar{\mathfrak{R}})$,
 using germ ordering.
We assume that one of the following assumptions holds. 

\begin{assump}\label{ass:1}
$\boldmu \gegerm \boldnu$ (with a fixed $\delta>0$ as in Definition~\ref{def:ordering}) and there exists $T>0$ such that  $\max \big (\mathfrak{r}_{x,y}(t), \delta \big ) \le \overline{\mathfrak{r}}_{x,y}(t)$ for all $t \ge T$, 
$x \in X$ and $y \in Y$.
\end{assump}
Assumption~\ref{ass:1} holds for instance if the family $\overline{\mathfrak{R}}$ is tight and there exists
$T_1>0$ such that $\mathfrak{r}_{x,y}(t)\le \overline{\mathfrak{r}}_{x,y}(t)$ (for all $t \ge T_1$, $x \in X$ and $y \in Y$). Indeed
in this case the existence of $T_2>0$ such that $\delta \le \overline{\mathfrak{r}}_{x,y}(t)$ (for all $t \ge T_2$, $x \in X$ and $y \in Y$) is trivial.

A similar, alternate assumption is the following.
\begin{assump}\label{ass:2}
$\boldmu \gepgf \boldnu$ and 
there exists $T>0$ such that  $\mathfrak{r}_{x,y}(t) \le \overline{\mathfrak{r}}_{x,y}(t)$ for all $t \ge T$, $x \in X$ and $y \in Y$.
\end{assump}

It is clear that under Assumption~\ref{ass:1} or Assumption~\ref{ass:2}, for all $t \ge T$ and $y \in Y$, we have
\begin{equation}
    \label{eq:assgerm}
G_\boldmu(\mathbf{\mathfrak{r}}_y(t)|y) \le G_\boldmu(\mathbf{\overline{\mathfrak{r}}}_y(t)|y)
\le G_\boldnu(\mathbf{\overline{\mathfrak{r}}}_y(t)|y).
\end{equation}

The above equation is the key to compare the behaviors of two processes
(recall that $\mathbf{\mathfrak{r}}_y(t,x):=\mathfrak{r}_{x,y}([0,t))$).
We denote by $\theta_\boldmu$, $\mathcal{N}_\boldmu$ and $\pr_\boldmu$ the quantities related to the 
$(\boldmu, {\mathfrak{R}})$
(to avoid a cumbersome notation we ignore the dependence on  $\mathfrak{R}$
of $\pr_\boldmu$). 
Similarly
we denote by $\theta_\boldnu$, $\mathcal{N}_\boldnu$ and $\pr_\boldnu$ the quantities related to 
$(\boldnu, \bar{\mathfrak{R}})$.

\begin{teo}
    \label{pro:comparehomfirework}
    Let $(\boldmu, {\mathfrak{R}})$ and $(\boldnu, \bar{\mathfrak{R}})$
    be two
    homogeneous firework processes on $\mathbb{N}$.
    Suppose that Assumption~\ref{ass:1} or Assumption~\ref{ass:2} holds.
    \begin{enumerate}
            \item 
        If
        $\theta_\boldnu(\mathbf{f}\colon \pr_\boldnu(V|\mathcal{N}_\boldnu=\mathbf{f})>0)=1-\nu(\mathbf{0})$ 
        then $\theta_\boldmu(\mathbf{f}\colon \pr_\boldmu(V|\mathcal{N}_\boldmu=\mathbf{f})>0)=1-\mu(\mathbf{0})$. 
        \item If $\theta_\boldmu(\mathbf{f}\colon \pr_\boldmu(V|\mathcal{N}_\boldmu=\mathbf{f})=0)=1$, then
        $\theta_\boldnu(\mathbf{f}\colon \pr_\boldnu(V|\mathcal{N}_\boldnu=\mathbf{f})=0)=1$.
            \end{enumerate}
\end{teo}

We note that, according to Theorem~\ref{thm:2}, either $\theta_\boldmu(\mathbf{f}\colon \pr_\boldmu(V|\mathcal{N}_\boldmu=\mathbf{f})=0)=1$ or $\theta_\boldmu(\mathbf{f}\colon \pr_\boldmu(V|\mathcal{N}_\boldmu=\mathbf{f})>0)=1-\mu(\mathbf{0})$.

\begin{teo}\label{pro:comparereverse} 
Let $(\boldmu, {\mathfrak{R}})$ and $(\boldnu, \bar{\mathfrak{R}})$
    be two
   reverse firework processes on $\mathbb{N}$. 
    Suppose that Assumption~\ref{ass:1} or Assumption~\ref{ass:2} holds.
\begin{enumerate}
 \item If $\theta_\boldnu(\mathbf{f}\colon \pr_\boldnu(S|\mathcal{N}_\boldnu=\mathbf{f})=1)=1-\nu_0(\mathbf{0})$ 
 then
 $\theta_\boldmu(\mathbf{f}\colon \pr_\boldmu(S|\mathcal{N}_\boldmu=\mathbf{f})=1)=1-\mu_0(\mathbf{0})$. 
\item If the two processes are homogeneous, we have that if $\theta_\boldmu(\mathbf{f}\colon \pr_\boldmu(S|\mathcal{N}_\boldmu=\mathbf{f})=0)=1$, then
$\theta_\boldnu(\mathbf{f}\colon \pr_\boldnu(S|\mathcal{N}_\boldnu=\mathbf{f})=0)=1$.
\end{enumerate}
\end{teo}
We stress that Theorem~\ref{pro:comparereverse} (1) holds both for the homogeneous and the inhomogeneous case.
We note also that, according to Theorem~\ref{thm:3}, either $\theta_\boldmu(\mathbf{f}\colon \pr_\boldmu(S|\mathcal{N}_\boldmu=\mathbf{f})=0)=1$ or $\theta_\boldmu(\mathbf{f}\colon \pr_\boldmu(S|\mathcal{N}_\boldmu=\mathbf{f})=1)=1-\mu_0(\mathbf{0})$.

The following examples deal with firework processes. We note that the same construction and results apply to the case of reverse firework processes (defined
with the same law $(\boldmu,\mathfrak{R})$).
In Example~\ref{exmp:fireworkmixedpoisson1} we compare two processes with only one type of stations. One process has a Poisson number of stations at each
site, while the second has a mixed Poisson, which is a perturbation of the total law of the first process. We show that the second process
can survive even if there are infinitely many sites which are ``subcritical" and it may go extinct even if there are infinitely many sites which are supercritical.
\begin{exmp}
    \label{exmp:fireworkmixedpoisson1}
    Let us consider two homogeneous firework processes $(\boldmu, {\mathfrak{R}})$ and $(\boldnu, \bar{\mathfrak{R}})$ on $\mathbb{N}$, where $\boldmu$ and $\boldnu$ are as in Example~\ref{exmp:mix}, $Y=\mathbb{N}$ and $X$ is a singleton. The total number of stations at each site is a $\mathcal{P}oi(\Lambda)$ random variable in the first process and $\mathcal{P}oi(\lambda)$ in the second process, where
    $\Lambda \sim \alpha \delta_{\lambda-\varepsilon}+(1-\alpha) \delta_{\lambda+\varepsilon}$ for some $\lambda >0$, $\varepsilon \in (0, \lambda)$ and $\alpha \in (0,1/2)$. As shown in Example~\ref{exmp:mix} we have $\boldmu \gegerm \boldnu$.
    Here, there is just one type of station with radius law $\mathfrak{r}$ (whence $\mathbf{\mathfrak{r}}$ is a 1-dimensional vector).
   Suppose that 
    $\mathfrak{r} \big ([i+1,+\infty) \big ) =\frac{1+o(1)}{\lambda_0(i+1)}$ as $i \to +\infty$ for some $\lambda_0 >0$,
    and that
    $\overline{\mathfrak{r}}=\mathfrak{r}$. 
    It is trivial to check that 
    \begin{equation}
        \label{eq:asympt}
        \sum_{n \in \mathbb{N}} \prod_{i=0}^n \exp \Big ( -\lambda \mathfrak{r}\big ([i+1,+\infty) \big ) \Big ) =\sum_{n \in \mathbb{N}} \exp \Big ( -\lambda \sum_{i=0}^n \mathfrak{r}\big ([i+1,+\infty) \big ) \Big ) 
    \begin{cases}
        < +\infty & \text{if } \lambda >\lambda_0 \\
        &\\
        = +\infty & \text{if } \lambda <\lambda_0\\
    \end{cases}
    \end{equation}
    since $\sum_{i=0}^n \mathfrak{r}\big ([i+1,+\infty) \big ) =(1+o(1)) \ln (n)/\lambda_0$ as $n \to +\infty$. 
    Recall the explicit expression of the generating function of a $\mathcal{P}oi(\Lambda_1)$ law, that is, $\E \big [\exp \big ((t-1) \Lambda_1 \big ) \big ]$.
According to Theorem~\ref{thm:2}, if $\lambda > \lambda_0$,
$(\boldnu, \bar{\mathfrak{R}})$
survives for almost every realization of the environment such that there is at least one station at $0\in\N$, while if $\lambda<\lambda_0$ the process dies out for almost all realizations of the environment. Suppose that $\lambda > \lambda_0> \lambda-\varepsilon$. According to Theorem~\ref{thm:2} and Theorem~\ref{pro:comparehomfirework},
$(\boldmu, {\mathfrak{R}})$
survives for almost all realizations of the environment such that there is at least one station at $0\in\N$. However, in almost every realization of the environment (generated according to $\boldmu$) there are arbitrarily long stretches of consecutive vertices where the number of stations is generated with a Poisson distribution of parameter $\lambda-\varepsilon$ (which is subcritical).
\end{exmp}

In the previous example there is just one type of stations. We now describe a situation where there are two types of stations.
In Example~\ref{ex:nonid} and Example~\ref{ex:nonidMP} we have a non-independently diffusing family $\boldmu$; the difference is that in the first one
the total distribution $\rho_y$ (which does not depend on $y$) is Poisson, while in the second example it is a mixed Poisson.
Example~\ref{ex:nonid} shows that,  when the number of stations is supercritical for
one of the types and subcritical for the other, if we randomly assign one single type to all stations at each site, we may have extinction; while if
we distribute the type of stations so that one half of them is of each type, then we may have survival. 

\begin{exmp}\label{ex:nonid}
Consider a family $\mathfrak{R}:=\{\mathfrak{r}_{x,y}\}_{x\in x,y\in Y}$, where
$X=\{1,2\}$, $Y=\N$, and $\mathfrak{r}_{i,n}=\mathfrak{r}_{i,m}$, for all $i=1,2$, $n,m\in\N$. Moreover, suppose that
$\mathfrak{r}_{i,1}([t,+\infty))=(1+o(1))/(\lambda_{i,0}\cdot t)$, for $i=1,2$, where
$o(1)\to0$ as $t\to\infty$ and $\lambda>0$. 
Let $\lambda_0:=\big(\beta/\lambda_{1,0}+(1-\beta)/\lambda_{2,0}\big)^{-1}$.
Note that $\lambda_{0}\in(\lambda_{1,0},\lambda_{2,0})$.
Consider now a family $\boldmu_2$, where $\mu_{2,n}=\mu_2$ for all $n\in\N$ is such that, at a given $n\in\N$, we generate  a $\mathcal{P}oi(\lambda)$-distributed
number of stations, which are all of type 1 with probability $\beta\in(0,1)$, and of type 2 with probability $1-\beta$.
Let $\phi$ be the generating function of $\mathcal{P}oi(\lambda)$. If $\lambda<\lambda_{0}$, then
the process $(\boldmu_2,\mathfrak{R})$ dies out for almost every configuration. Indeed, let $\boldnu$ be the homogeneous, independently diffusing family,
given by the total distribution $\mathcal{P}oi(\lambda)$ and matrix $\{P_{n,i}\}_{n\in\N, i=1,2}$, with $p_{n,1}=\beta$, $p_{n,2}=1-\beta$.
We know from Section~\ref{subsec:genveindep} that, for all $\mathbf z\in[0,1]^\N$,
\[
G_{\boldmu_2}(\mathbf{z}|n)=\beta \phi\big(\mathbf{z}(1)\big)+(1-\beta)\phi\big(\mathbf{z}(2)\big),\qquad 
G_{\boldnu}(\mathbf{z}|n)= \phi\big(\beta\mathbf{z}(1)+(1-\beta)\mathbf{z}(2)\big),
\]
and $\boldnu\gepgf\boldmu_2$. According to Remark~\ref{rem:novelty}, the firework process $(\boldnu,\mathfrak{R})$ is equivalent to a single-type
process with total distribution $\mathcal{P}oi(\lambda)$ and radii distribution $\mathfrak{r}$ which does not depend on the site and is defined by
\[
\mathfrak{r}([t,+\infty):=(\beta\mathfrak{r}_{1,n}+(1-\beta)\mathfrak{r}_{2,n})([t,+\infty))=\frac{1+o(1)}{t}\lambda_0.
\]
By Equation~\eqref{eq:asympt}, if $\lambda<\lambda_{0}$, then $(\boldnu,\mathfrak{R})$ dies out for almost all
configurations and, by Theorem~\ref{pro:comparehomfirework}, so does $(\boldmu_2,\mathfrak{R})$.

Let now $\boldmu_1$ be as in Section~\ref{subsec:genveindep}, where $\rho_y\sim \mathcal{P}oi(\lambda)$ does not depend on $y$,
and the sequence $\alpha_i=\lfloor i/2\rfloor$. Roughly speaking, at any $y\in\N$, a $\mathcal{P}oi(\lambda)$-distributed number of stations is generated:
$\lfloor i/2\rfloor$ are of type 1, $\lfloor i/2\rfloor$ are of type 2 and the remainder (if there is any), is randomly assigned one of the types in $\{1,2\}$, each
with probability $1/2$.
If $\lambda>\lambda_0$, then
the firework process $(\boldmu_1,\mathfrak{R})$ survives with positive probability for almost every configuration, with at least one active station at
the origin.
Indeed, from Section~\ref{subsec:genveindep}, we know that $\boldmu_1\gepgf\boldnu$, where $\boldnu$ is as before, with $\beta=1/2$. 
Again, by Equation~\eqref{eq:asympt}, if $\lambda>\lambda_{0}$, then $(\boldnu,\mathfrak{R})$ 
survives with positive probability, for almost every configuration which has at least one active station at
the origin.
By Theorem~\ref{pro:comparehomfirework}, so does $(\boldmu_1,\mathfrak{R})$.
\end{exmp}

\begin{exmp}\label{ex:nonidMP}
Let law of the radii $\mathfrak{R}$, $\boldmu_1$, $\boldmu_2$  and the parameter $\lambda_0$ be
 as in Example~\ref{ex:nonid}.
 We change the common total distribution of $\boldmu_1$, $\boldmu_2$: instead of $\mathcal{P}oi(\lambda)$ we
 take $\rho_y\sim\mathcal{P}oi(\Lambda)$ (as in Example~\ref{exmp:mix}), where $\Lambda\sim\alpha\delta_{\lambda_1}+(1-\alpha)\delta_{\lambda_2}$,
$\alpha\in(0,1)$, $0<\lambda_1<\lambda_2$.
If
$\lambda_1<\lambda_0$ and $\alpha>(\lambda_2-\lambda_0)/(\lambda_2-\lambda_1)$, then $(\boldmu_2,\mathfrak{R})$
dies out for almost every configuration.
If
$\lambda_1>\lambda_0$ and $\alpha<(\lambda_2-\lambda_0)/(\lambda_2-\lambda_1)$, then $(\boldmu_1,\mathfrak{R})$
survives with positive probability for almost every configuration, with at least one active station at the origin.

We sketch the proof of the first case, when there is extinction (the case of survival is similar).
Define $\boldnu$ as in Example~\ref{ex:nonid}, with the mixed Poisson total distribution.
Clearly, $\boldnu\gepgf\boldmu_2$. Let us take $\lambda>0$ such that $\lambda_1<\lambda<\lambda_0$ and
$\alpha>(\lambda_2-\lambda)/(\lambda_2-\lambda_1)$. Consider the family of independently diffusing measures $\bar\boldnu$ defined as $\boldnu$,
with total distribution $\mathcal{P}oi(\lambda)$ instead of $\mathcal{P}oi(\Lambda)$.
Since $\alpha>(\lambda_2-\lambda)/(\lambda_2-\lambda_1)$, we have that, according to Example~\ref{exmp:mix}, $\bar\boldnu\gepgf\boldnu$.
From $\lambda<\lambda_0$, we get that $(\bar\boldnu,\mathfrak{R})$ dies out for almost every configuration, see Equation~\eqref{eq:asympt}
(see also Theorem~\ref{thm:2}).
We apply Theorem~\ref{pro:comparehomfirework} twice: we have that $(\boldnu,\mathfrak{R})$ and  $(\boldmu_2,\mathfrak{R})$ die out for almost every configuration.
\end{exmp}

\section{Proofs}\label{sec:proofs}

\begin{proof}[Proof of Proposition~\ref{pro:Gtopology}]
  
	\leavevmode
\begin{enumerate}
\item It is easy: see \cite[Sections 2 and 3]{cf:BZ2} for the details.    
\item It is enough to prove that $\mathbf{z} \mapsto G_\boldmu(\mathbf{z}|y)=\sum_{f \in S_X} \mu_y(f) \prod_{x \in X} \mathbf{z}(x)^{f(x)}$ is continuous with respect to the pointwise convergence topology for each $y \in Y$. To this aim, note that $\sup_{\mathbf{z} \in [0,1]^X} \big |\mu_y(f) \prod_{x \in X} \mathbf{z}(x)^{f(x)} \big |=\mu_y(f)$ and $\mathbf{z} \mapsto \prod_{x \in X} \mathbf{z}(x)^{f(x)}$ is continuous with respect to the pointwise convergence topology for every $f \in S_X$. Since $\sum_{f \in S_X} \mu_y(f)=1<+\infty$ then $\sum_{f \in S_X} \mu_y(f) \prod_{x \in X} \mathbf{z}(x)^{f(x)}$ converges to $G_\boldmu(\mathbf{z}|y)$ uniformly with respect to $\mathbf{z} \in [0,1]^X$, therefore $G_\boldmu(\cdot|y)$ is continuous.
\item 
Let us note that the tightness of $\{\rho_y\}_{y \in Y}$ means that for every $\varepsilon >0$ there exists $n=n(\varepsilon) \in \mathbb{N}$ such that $\mu_y(f\colon |f| \le n) > 1-\varepsilon $ for all $y \in Y$.
We start by proving that $\big | \prod_{x \in X} \mathbf{z}(x)^{f(x)}-\prod_{x \in X} \mathbf{v}(x)^{f(x)} \big | \le \min(1, |f| \cdot \|\mathbf{z}-\mathbf{v}\|_\infty)$. The inequality $\big | \prod_{x \in X} \mathbf{z}(x)^{f(x)}-\prod_{x \in X} \mathbf{v}(x)^{f(x)} \big | \le 1$ follows trivially from the fact that $|t-s| \le \max(|t|,|s|)$ for $t,s \ge 0$. As for the second inequality, observe that if $t_i, s_i \in [0,1]$ for all $i=1, \ldots,n$ then
\[
\begin{split}
    \prod_{i=1}^n t_i-\prod_{i=1}^n s_i =&
    \prod_{i=1}^n t_i-s_n \prod_{i=1}^{n-1} t_i+ s_n \prod_{i=1}^{n-1} t_i+ \cdots +\prod_{i=1}^{j} t_i\prod_{i=j+1}^{n} s_i - \prod_{i=1}^{j-1} t_i\prod_{i=j}^{n} s_i\\ &+ \cdots + t_1 \prod_{i=2}^n s_i -  \prod_{i=1}^n s_i
\end{split}
\]
whence
\[
\begin{split}
    \Big |\prod_{i=1}^n t_i-\prod_{i=1}^n s_i\Big | \le &
    \Big |\prod_{i=1}^{n-1} t_i\Big | \cdot |t_n-s_n|+ \cdots +
    \Big |\prod_{i=1}^{j-1} t_i\prod_{i=j+1}^{n} s_i|\cdot |t_j-s_j|\\ &+ \cdots + \Big |\prod_{i=2}^n s_i\Big |\cdot |t_1-s_1| \le n \max_{i=1, \ldots, n} |t_i-s_i|.
\end{split}
\]
Let us fix $\varepsilon >0$ and let $n=n(\varepsilon/2)$ (coming from the tightness). For every $\mathbf{z}, \mathbf{v}$ such that $\|\mathbf{z}-\mathbf{v}\|_\infty \le \delta:=\varepsilon/(2n)$ and for every $y \in Y$ we have
\[
\begin{split}
    \big |G_\boldmu(\mathbf{z}|y)-G_\boldmu(\mathbf{v}|y) \big | &\le 
\sum_{f \in S_X\colon |f| \le n} \mu_y(f) \Big |\prod_{x \in X}  \mathbf{z}(x)^{f(x)} - \prod_{x \in X}\mathbf{v}(x)^{f(x)} \Big |\\&+
\sum_{f \in S_X\colon |f| >n} \mu_y(f) \Big |\prod_{x \in X}  \mathbf{z}(x)^{f(x)} - \prod_{x \in X}\mathbf{v}(x)^{f(x)} \Big |\\
&\le \sum_{f \in S_X\colon |f| \le n}  \mu_y(f) |f|\cdot \|\mathbf{z}-\mathbf{v}\|_\infty+ \sum_{f \in S_X\colon |f| >n} \mu_y(f)
\\
& \le n\frac{\varepsilon}{2n}+\frac{\varepsilon}{2}=\varepsilon.
\end{split}
\]
Hence $\|\mathbf{z}-\mathbf{v}\|_\infty \le \delta$ implies $\big \|G_\boldmu(\mathbf{z})-G_\boldmu(\mathbf{v}) \big \|_\infty \le \varepsilon$.
  \end{enumerate}
\end{proof}

\begin{proof}[Proof of Theorem~\ref{pro:germindepdiff}]
	\noindent $(1) \Rightarrow (2)$. 
	Using the hypothesis and the expression for $\phi_y$, we get that
	for every $t \in [\delta,1]$ and for all $x \in X$, since $t \mathbf{1} \in [\delta,1]^X$ then
	\[
	\phi_y^\boldmu(t)
	=G_\boldmu(t \mathbf{1}|y) \le G_\boldnu(t \mathbf{1}|y)
	=\phi_x^\boldnu(t).
	\]
	
	\noindent $(2) \Rightarrow (1)$. 
	Recall that, for independently diffusing families, the generating functions are
	$G_\boldmu(\mathbf{z}|y)=\phi_y^\boldmu(P\mathbf{z}(y))$ and
	$G_\boldnu(\mathbf{z}|y)=\phi_y^\boldnu(P\mathbf{z}(y))$.
	We observe that the map $\mathbf{z} \mapsto P\mathbf{z}$
	is non-decreasing and continuous from $[0,1]^X$ into $[0,1]^Y$; in
	particular, if $\mathbf{z} \in [\delta,1]^X$ for some $\delta<1$, then $P\mathbf{z} \in [\delta,1]^Y$. Indeed
	$P\, t \mathbf{1}=t \mathbf{1}$ therefore $\delta \mathbf{1}=P\, \delta \mathbf{1} \le P \mathbf{z} \le P \mathbf{1}=\mathbf{1}$.
	Take $\mathbf{z} \in [\delta,1]^X$; then for all $y \in Y$
	\[
	G_\boldmu(\mathbf{z}|y)=\phi_y^\boldmu(P\mathbf{z}(y)) \le
	\phi_y^\boldnu(P\mathbf{z}(y))=
	G_\boldnu(\mathbf{z}|y)
	\]
	where we used the inequality $\phi_y^\boldmu(t) \le \phi_y^\boldnu(t)$ for $t=P\mathbf{z}(y) \in [\delta,1]$ (due to the monotonicity of $P$).
\end{proof}

\begin{proof}[Proof of Theorem~\ref{pro:germidextended}]
	\leavevmode
\begin{enumerate}
\item
By the monotone convergence theorem,
$\E[U_y \cdot \exp(-t U_y)] \uparrow \E[U_y)] \in (\E[W_y],+\infty]$, as $t\uparrow 0$.
Therefore,
 there exists $\delta=\delta(y)<1$ such that 
$\E[U_y\cdot \exp(-t U_y)] \ge\E[W_y]$, for all $t\in[\delta,1]$.
\item
We note that there exists $\delta<1$ such that $\E[t^{U_y}]\le \E[t^{W_y}]$ for all $t \in [\delta,1]$ if and only if there exists $\delta>0$ such that $\E[\exp(-t U_y)] \le \E[\exp(-t W_y)]$ for all $t \in [0,\delta]$. Moreover, by using well-known results in measure theory
\[
\frac{\diff ^k}{\diff t^k} \E[ \exp(-t U_y)- \exp(-t W_y)] =(-1)^k \E[U^k_y \exp(-t U_y)-W^k_y \exp(-t W_y)]
\]
since $Z^k\exp(-tZ) \le (k/\varepsilon e)^k <+\infty$ uniformly w.~r.~to $t \in [\varepsilon, +\infty)$ for every $\varepsilon>0$, 
every $k\in\N$ and some $M>0$ ($Z$ is a generic non-negative random variable). The same inequality holds in $t=0$ by taking the limit when $t \to 0^-$.
More generally, given a generic non-negative random variable $Z$, the Bounded Convergence theorem yields
$\sum_{i=1}^n Z^i \exp(-tZ)/i! \uparrow \exp((1-t)Z)$ a.s.~(and in $L^1$ whenever $\E[\exp((1-t)Z)]<+\infty$, for instance for all $t \ge 1$). Clearly $\E[\sum_{i=1}^n Z^i \exp(-tZ)/i!]=\sum_{i=1}^n \E[Z^i \exp(-tZ)]/i! \uparrow \E[\exp(1-t)Z]$.

Define now $H_y(t):= \E[ \exp(-t U_y)- \exp(-t W_y)]$; under the hypotheses, by using Taylor expansion in a right neighborhood of $0$, if $t >0$
\[
H_y(t)=\frac{(-1)^{k+1}}{(k+1)!} \E[U^{k+1}_y \exp(-s U_y)-W^{k+1}_y \exp(-s W_y)]
\]
for some $s \in (0,t)$.
Now if $t \in (0,\delta]$ and $k$ is even then
\[
\E[U^{k+1}_y \exp(-s U_y)] \ge \E[U^{k+1}_y \exp(-\delta U_y)] \ge \E[W^{k+1}_y] \ge  \E[W^{k+1}_y \exp(-s U_y)]
\]
whence $H_y(t) \le 0$. The same happens if $k$ is odd by using similar arguments.
Theorem~\ref{pro:germindepdiff} yields the claim, since by assumption $\boldmu$ and $\boldnu$ are independently diffusing families with the same matrix $P$.
\end{enumerate}
\end{proof}

\begin{proof}[Proof of Theorem~\ref{pro:betterconditions}]
The proof 
is similar to the proof of Theorem~\ref{pro:germidextended}. The basic fact here is that if we define $H(t):=\E[t^{U_y}-t^{W_y}]$, then elementary results of measure theory yields
	\[
	\frac{\diff^k}{\diff t^k} H_y(t)=\E[t^{U_y-k}\prod_{j=0}^{k-1} (U_y-j)]
	\]
The proof then goes on as in the proof of Theorem~\ref{pro:germidextended}, using Taylor expansion in a left neighborhood of $1$.

Define $\varphi^\boldmu_y(t):=\phi^\boldmu_y(\exp(-t))$, $\varphi^\boldnu_y(t):=\phi^\boldnu_y(\exp(-t))$ and $\delta_1:=\exp(-\delta)$. Let us fix $k=0$: Equation~\eqref{eq:condvarphi} is equivalent to $\frac \diff{\diff t}\varphi^\boldnu_y(t)|_{t=\delta}\le\frac \diff{\diff t}\varphi^\boldmu_y(t)|_{t=0}$, while 
\eqref{eq:condphi} is equivalent to $\frac \diff{\diff t}\phi^\boldnu_y(t)|_{t=\delta_1}\ge\frac \diff{\diff t}\phi^\boldmu_y(t)|_{t=1}$. However
\[\begin{split}
    \frac \diff{\diff t}\phi^\boldnu_y(t)|_{t=\delta_1}\ge\frac \diff{\diff t}\phi^\boldmu_y(t)|_{t=1} \Longleftrightarrow &
\frac \diff{\diff t}\varphi^\boldnu_y(t)|_{t=\delta}\le
\exp(-\delta)\frac \diff{\diff t}\varphi^\boldmu_y(t)|_{t=0}\\
 & \qquad \qquad \Uparrow \qquad \not\Downarrow\\
 & \frac \diff{\diff t}\varphi^\boldnu_y(t)|_{t=\delta}\le\frac \diff{\diff t}\varphi^\boldmu_y(t)|_{t=0},
\end{split}
\]
since the derivatives of the $\varphi_y$s are negative. Therefore 
\eqref{eq:condphi} implies, and is not implied by, \eqref{eq:condvarphi}, when $k=0$. The cases $k\ge1$ are analogous.
\end{proof}

\begin{proof}[Proof of Proposition~\ref{pro:comparemix}]
The key to the proof is the explicit expression of the generating functions for mixed Poisson.
Let $U_y\sim\rho^\boldmu_y$ and  $W_y\sim\rho^\boldnu_y$.
By elementary measure theory, 
their generating functions can be explicitly computed as
are
\[\begin{split}
\phi_{U_y}(t)=\sum_{j\in\N} t^j \E[\exp(-\Lambda_{1,y})\cdot \Lambda_{1,y}^j/j!]=\E[\exp((t-1)\Lambda_{1,y})], \quad 
\phi_{W_y}(t) =\E[\exp((t-1)\Lambda_{2,y})];
\end{split}
\]
for $t \in[0,1]$ and their derivatives are
\[\begin{split}
\frac \diff{\diff t}\phi_{U_y}(t)=\E[\Lambda_{1,y}\cdot\exp((t-1)\Lambda_{1,y})],\quad 
\frac \diff{\diff t}\phi_{W_y}(t)=\E[\Lambda_{2,y}\cdot\exp((t-1)\Lambda_{2,y})]
\end{split}
\]
for $t \in[0,1]$ (where the derivative in $1$ is the left derivative).

	Elementary computations show that the condition of Proposition~\ref{pro:comparemix} is the case $k=0$ of the conditions discussed in Theorem~\ref{pro:betterconditions},
    therefore the proof is similar to the proof of Theorem~\ref{pro:germidextended}. 
\end{proof}

\begin{proof}[Proof of Example~\ref{exmp:mix2}]
	We aim at applying Proposition~\ref{pro:comparemix} (2). Let 
    $\alpha:=\sup_y\alpha_y$: by hypothesis, $\alpha<1/2$. We evaluate, for any $\delta<1$,
\[
\begin{split}
    \E[\Lambda_{1,y}\cdot e^{(\delta-1)\Lambda_{1,y}}]& =
    \alpha_y(\lambda_y-\varepsilon)e^{(\lambda_y-\varepsilon)(\delta-1)}
    + (1-\alpha_y)(\lambda_y+\varepsilon)e^{(\lambda_y+\varepsilon)(\delta-1)}\\
    &\ge
    \big(\alpha_y(\lambda_y-\varepsilon)+(1-\alpha_y)(\lambda_y+\varepsilon)    \big)e^{(\lambda_y+\varepsilon)(\delta-1)}   \\
    &=
    \big((\lambda_y+(1-2\alpha_y)\varepsilon   \big)e^{(\lambda_y+\varepsilon)(\delta-1)} \\
    & \ge 
    \big((\lambda_y+(1-2\alpha)\varepsilon   \big)e^{(M+\varepsilon)(\delta-1)}.   
\end{split}
\]    
Take $\delta<1$ such that 
\[
e^{(M+\varepsilon)(\delta-1)}>\frac{M}{M+ (1-2\alpha)\varepsilon}\ge 
\frac{\lambda_y}{\lambda_y+ (1-2\alpha)\varepsilon}.
\]
By this choice,
\[
\E[\Lambda_{1,y}\cdot e^{(\delta-1)\Lambda_{1,y}}]\ge
     \big((\lambda_y+(1-2\alpha)\varepsilon   \big)e^{(M+\varepsilon)(\delta-1)}\ge \lambda_y =  \E[\Lambda_{2,y}],
\]  
whence, by Proposition~\ref{pro:comparemix}, $\boldmu\gegerm\boldnu$.

If there exists $y\in Y$ such that $\alpha_y>(1-\exp(-\varepsilon))/(\exp(\varepsilon)-\exp(-\varepsilon))$, then
\[
\E[\exp(-\Lambda_{1,y})]=\alpha_y e^{-(\lambda_y-\varepsilon)}
    + (1-\alpha_y)e^{-(\lambda_y+\varepsilon)} >\exp(-\lambda_{y}),
    \]
    which implies $\boldmu \not \!\! \gepgf \boldnu$.

    Conversely, let $\alpha:=\inf_y \alpha_y>1/2$. If $\delta<1$
    \[
\begin{split}
\E[\Lambda_{2,y}e^{\Lambda_{2,y}(\delta-1)}-\Lambda_{1,y}]&=\lambda_y e^{\lambda_y(\delta-1)} - \alpha_y {(\lambda_y-\varepsilon)}-(1-\alpha_y){(\lambda_y+\varepsilon)} \\
&= \varepsilon (2\alpha_y-1)-\lambda_y(1-e^{\lambda_y(\delta-1)}) \ge 
\varepsilon (2\alpha-1)-M(1-e^{M(\delta-1)})\ge 0
\end{split}    
    \]
    for all $\delta \ge 1-\delta_0$ for a suitable $\delta_0<1$.
\end{proof}

\begin{proof}[Proof of Proposition~\ref{pro:BPVE1}]
We prove below that, under the hypotheses, inequality~\eqref{eq:reverseinequality2} holds, that is, there exists $\delta \in [0,1)$ such that for all sufficiently large $n \in \mathbb{N}$ 
\[
\sum_{k=0}^{k_0} \delta^k \Big (\sum_{i \ge k+1} \rho^\boldmu_n(i) \Big ) \ge a_n.
\]
Then we know by Lemma~\ref{lem:donotwriteanotherlemmaplease} that $\phi^\boldmu_n(t) \le \phi^\boldnu_n(t)$ for all $t \in [\delta, 1]$, whence Proposition~\ref{pro:germBPVE} yields the result (as in Example~\ref{exmp:BPVE1}).

Note that, by hypotheses, there exists $\Delta >0$ such that, for all sufficiently large $n \in \mathbb{N}$, $\sum_{k=0}^{k_0} (k_0+1-k) \rho^\boldmu_n(k) \le k_0-\Delta$. Therefore,
if $\delta \in \big ((1+\Delta)^{-1/k_0}, 1 \big )$, for all sufficiently large $n \in \mathbb{N}$ we have
\[
\begin{split}
\sum_{k=0}^{k_0} \delta^k \Big (\sum_{i \ge k+1} \rho^\boldmu_n(i) \Big )& =
\sum_{k=0}^{k_0} \delta^k \Big (1-\sum_{i=0}^k \rho^\boldmu_n(i) \Big )
\ge \delta^{k_0}\sum_{k=0}^{k_0}\Big (1-\sum_{i=0}^k \rho^\boldmu_n(i) \Big )\\
& 
=
\delta^{k_0}\Big ( k_0+1-\sum_{k=0}^{k_0}\sum_{i=0}^k \rho^\boldmu_n(i) \Big )
=\delta^{k_0}\Big ( k_0+1-\sum_{i=0}^{k_0}\sum_{k=i}^{k_0} \rho^\boldmu_n(i) \Big )\\
&= \delta^{k_0}\Big ( k_0+1-\sum_{i=0}^{k_0} (k_0+1-i) \rho^\boldmu_n(i) \Big )
\ge \delta^{k_0} (1+\Delta) \ge 1>a_n.
\end{split}
\]
\end{proof}

The proofs of the results of Section~\ref{sec:firework} follow closely the proofs of \cite{cf:BZ13}.

\begin{proof}[Proof of Example~\ref{exmp:nonind1}]
The explicit expression of $\mu_{1,y}$ is
\[
\mu_{1,y}(f)=
\begin{cases}
    \rho_y(i) \frac{r_i!}{\prod_{x \in X} (f(x)-\alpha_i k p_x)!}
    \prod_{x \in X} p(y,x)^{f(x)-\alpha_i k p_x} & \textrm{if } f(x) =r_i+ \alpha_i k p_x,\, \forall x \in X\\
    0 & \textrm{otherwise.}
\end{cases}
\]
It is easy to show that $\prod_{x \in X} \mathbf{z}(x)^{p_x} \le  \sum_{x \in X} p_x \mathbf{z}(x)$ (and the equality holds if and only if $z(x)$ does not depend on $x \in X$ such that $p_x>0$), whence
\[
\begin{split}
G_{\boldmu_1}(\mathbf{z}|y)&=\sum_{i=0}^\infty 
\rho_y(i) \Big ( \prod_{x \in X} \mathbf{z}(x)^{\alpha_i k p_x}  \Big ) \Big ( \sum_{x \in X} p_x \mathbf{z}(x)  \Big )^{r_i} \\
& \le \sum_{i=0}^\infty 
\rho_y(i) \Big ( \sum_{x \in X} p_x \mathbf{z}(x)  \Big )^{k\alpha_i} \Big ( \sum_{x \in X} p_x \mathbf{z}(x)  \Big )^{r_i} \\
& = \sum_{i=0}^\infty \rho_y(i) \Big ( \sum_{x \in X} p_x \mathbf{z}(x)  \Big )^{i} = G_\boldnu( \mathbf{z} |y).
\end{split}
\]
This proves that $\boldmu_1  \gepgf \boldnu$, which implies $\boldmu_1  \gegerm \boldnu$.
\end{proof}

\begin{proof}[Proof of Theorem~\ref{thm:2}]
    We investigate the behavior of the single-station counterpart of our process.
Since $R_{n,x,i} \sim \mathfrak{r}_x$ and $F_i \sim \mu$ for all $n,x,i$ then the law of 
$\widetilde R_{n} \sim \widetilde R$ does not depend on $n \in \N$ and, according to Equation~\eqref{eq:radiusmaxlaw}, $\pr(\widetilde R < i)= G_\boldmu(\mathfrak{r}(i)|0)$.

By  \cite[Theorem 2.1]{cf:JMZ}
the a.s.~extinction of the single-station counterpart is equivalent to
\[
 \sum_{n=0}^\infty \prod_{i=0}^n \pr(\widetilde R < i+1)=+\infty.
\]
Hence Equation~\eqref{eq:sumprod2.5} is equivalent to $\pr(V)=0$ which, in turn,
is equivalent to $\pr(V| \mathcal{N}=\mathbf{f})=0$ for $\theta$-almost all configurations $\mathbf{f} \in (S_X)^\N$ (see the discussiont after eq.~\eqref{eq:disintegration}).

We are left to prove that $\pr(V)>0$ implies $\theta(\mathbf{f}  \in {(S_X)^{\mathbb{N}}} \colon \pr(V|\mathcal{N}=\mathbf{f})>0)=\pr(F_0 \neq \mathbf{0})$, since $\pr(F_0 \neq \mathbf{0})=1-\mu(\mathbf{0})$.
Note that $\pr(V)=\pr(V| F_0 \neq \mathbf{0}) \pr(F_0 \neq \mathbf{0})$
since $\pr(V| F_0 = \mathbf{0})=0$; moreover, by our assumptions, $\pr(V| F_0 \neq \mathbf{0})>0$ if and only if
$\pr(V)>0$. We condition now on the event $\{F_0 \neq \mathbf{0}\}=\widehat{{(S_X)^{\mathbb{N}}}} \times \mathcal{O}$ (see Section~\ref{subsec:N}):
where we denote by $\widehat{{(S_X)^{\mathbb{N}}}}$ the space $\{\mathbf{f}=\{f_i\}_{i \in \N}\colon f_0 \neq \mathbf{0}\}$ and by $\widehat \theta$ the measure $\theta$
conditioned on $\widehat{{(S_X)^{\mathbb{N}}}}$ (clearly $\theta(\widehat{{(S_X)^{\mathbb{N}}}})=\pr(F_0 \neq \mathbf{0})$). Observe that $\theta(\mathbf{f}=\{f_i\}_{i \in \N}\colon
f_i \neq \mathbf{0}\ \textrm{infinitely often})=1$ 
and that
the variables $\{F_i
\}_{i \ge 1}$ are independent 
with respect to the conditioned
probability $\pr(\cdot| F_0 \neq \mathbf{0})$.
Denote by
\[
 W_{k}:=\{\mathbf{f}=\{f_i\}_{i \in \N} \in \widehat{{(S_X)^{\mathbb{N}}}}\colon  \pr(\textrm{``}V \textrm{ starting from }f_k \textrm{ stations at }
k\textrm{''}| \mathcal{N}=\mathbf{f})>0 \}
\]
the set of realizations of the environment such that the firework process starting from
vertex $k$ (and moving rightwards) survives with  positive probability.
Note that $\theta(W_0)=\theta(\mathbf{f} \in {(S_X)^{\mathbb{N}}} \colon  \pr(V| \mathcal{N}=\mathbf{f})>0 )$.
Clearly $W_k \in \sigma (F_i\colon i \ge k)$
where, in this case, $\{F_i\}_{i \in \N}$ is the canonical realization of $\mathcal{N}$ on ${(S_X)^{\mathbb{N}}}$.

If for a fixed sequence in $\widehat{{(S_X)^{\mathbb{N}}}}$  there is survival for the firework process
starting from $f_{i_0}$ stations at $i_0$ 
then,
 by the FKG inequality,  there is survival starting from every $i \le i_0$
(note that 
if Equation~\eqref{eq:sumprod2.5} holds then the radii are unbounded variables, hence
the event $\{\widetilde R>i_0\}$ has a positive
  probability). Whence $W_k \supseteq W_{k+1}$ for all $k \in \N$.

Moreover, if there is survival for a sequence $\mathbf{f}=\{f_i\}_{i \in \N}$ then for all $j \in \N$
  there exists $i_0 \ge j$ such that there is survival starting from
$f_{i_0}$ stations at $i_0$. Hence
$W_0 \subseteq \limsup_k W_k$ which implies that $W_0=\limsup_k W_k=W_i$ for all $i \in \N$.
Hence $W_0$ is a tail event, namely it belongs to $\bigcap_{k \in \N} \sigma (F_i\colon i \ge k)$.
Thus
$\widehat \theta(W_0)$ is either $0$ or $1$. This implies that
$\theta(W_0)$ is either $0$ or $\pr(F_0 \neq \mathbf{0})$.

Equation~\eqref{eq:sumprod2.5} implies $\pr(V)>0$ or, equivalently,
$\theta(\mathbf{f} \in {(S_X)^{\mathbb{N}}} \colon \pr(V| \mathcal{N}=\mathbf{f})>0 )>0$.
Since $\theta(W_0)=\theta(\mathbf{f} \in {(S_X)^{\mathbb{N}}} \colon \pr(V| \mathcal{N}=\mathbf{f})>0 )$
we have that $\theta(\mathbf{f} \in {(S_X)^{\mathbb{N}}} \colon \pr(V| \mathcal{N}=\mathbf{f})\neq \mathbf{0} )=\pr(F_0 \neq \mathbf{0})$.
\end{proof}

\begin{proof}[Proof of Theorem~\ref{th:reversehomogeneous}]
Since the radii of the stations at the origin do not play any role in the process, we can again condition on the event $\{F_0 \neq \mathbf{0}\}$ 
whose probability is 
$\pr(F_0 \neq \mathbf{0})=1-\mu_0(\mathbf{0})$. 
Let $\widehat{{(S_X)^{\mathbb{N}}}}$ and $\widehat \theta$ as in the previous proof.
Since $\pr(S|\mathcal{N}=\mathbf{f})=0$ for every $\mathbf{f} \not \in \widehat{{(S_X)^{\mathbb{N}}}}$, we are left to prove that, if there is at least one station at the origin, then
    \begin{enumerate}
      \item $W=+\infty \Longrightarrow \widehat \theta(\mathbf{f} \in A\colon \pr(S|\mathcal{N}=\mathbf{f})=1)=1$, 
\item $W<+\infty \Longrightarrow \widehat \theta(\mathbf{f} \in A\colon \pr(S|\mathcal{N}=\mathbf{f})=0)=1$.
    \end{enumerate}
To this aim, let us apply \cite[Theorem 2.8]{cf:JMZ} to the single-station process.
Trivially $\E[\widetilde R]<+\infty$ if and only if $W<+\infty$.
 The
results follow immediately from the equivalence between
$\pr(S|F_0 \neq \mathbf{0})=1$ (resp.~$\pr(S|F_0 \neq \mathbf{0})=0$) and
$\pr(S|\mathcal{N}=\mathbf{f})=1$ (resp.~$\pr(S|\mathcal{N}=\mathbf{f})=0$) for $\widehat \theta$-almost all configurations $\mathbf{f} \in \widehat{{(S_X)^{\mathbb{N}}}}$ which is an easy consequence of Equation~\eqref{eq:disintegration}
Indeed, since $\{F_0 \neq \mathbf{0}\}= \widehat{{(S_X)^{\mathbb{N}}}} \times \mathcal{O}$ we have that, for every event $A$,
\[
\begin{split}
    \pr(A|F_0 \neq \mathbf{0})&= \frac{\pr(A \cap \{F_0 \neq \mathbf{0}\})}{\pr(F_0 \neq \mathbf{0})}=\frac{1}{\theta(\widehat{{(S_X)^{\mathbb{N}}}})} \int_{S_X} \pr \big (A \cap (\widehat{{(S_X)^{\mathbb{N}}}} \times \mathcal{O})|\mathcal{N}=\mathbf{f}\big )  \theta(\diff \mathbf{f})\\
    &= \int_{S_X} \nu\big (\mathbf{r} \in \mathcal{O}\colon (\mathbf{f}, \mathbf{r}) \in A \cap (\widehat{{(S_X)^{\mathbb{N}}}} \times \mathcal{O})\big ) \widehat \theta(\diff \mathbf{f})
    = \int_{\widehat{{(S_X)^{\mathbb{N}}}}} \nu\big (\mathbf{r} \in \mathcal{O}\colon (\mathbf{f}, \mathbf{r}) \in A\big ) \widehat \theta(\diff \mathbf{f})\\
    &= \int_{\widehat{{(S_X)^{\mathbb{N}}}}} \pr(A|\mathcal{N}=\mathbf{f})\widehat \theta(\diff \mathbf{f})
\end{split}
\]
since
\[\nu\big (\mathbf{r} \in \mathcal{O}\colon (\mathbf{f}, \mathbf{r}) \in A \cap (\widehat{{(S_X)^{\mathbb{N}}}} \times \mathcal{O})\big )=
\begin{cases}
    \nu\big (\mathbf{r} \in \mathcal{O}\colon (\mathbf{f}, \mathbf{r}) \in A\big ) & \textrm{if } \mathbf{f} \in \widehat{{(S_X)^{\mathbb{N}}}}\\
    0 &  \textrm{if } \mathbf{f} \not \in \widehat{{(S_X)^{\mathbb{N}}}}.
\end{cases}
\]
Note that, for every event $A$, $\pr(A|\mathcal{N}=\mathbf{f},F_0 \neq \mathbf{0})=\pr(A|\mathcal{N}=\mathbf{f})$ for all $\mathbf{f} \in \widehat{{(S_X)^{\mathbb{N}}}}$.

\end{proof}

In order to prove Theorem~\ref{thm:3} we need a lemma. In the following, given $\mathbf{z} \in [0,1]^X$ and $f \in S_X$ we define $\mathbf{z}^f:=\prod_{x \in X} \mathbf{z}(x)^{f(x)}$. Recall that $\{F_n\}_{n \in \N}$ is a sequence of independent $S_X$-valued random variables such that $F_n \sim \mu_n$ for all $n \in \N$.

\begin{lem}\label{lem:1}
 If $\{\mathbf{t}_{i,n}\}_{i,n \in \N, i \le n}$ is an arbitrary sequence 
 $\mathbf{t}_{i,n} \in [0,1]^X$
 and
\begin{equation}\label{eq:sumprod}
 \sum_{n=0}^\infty \prod_{i=0}^n G_\boldmu(\mathbf{t}_{i,n}|i)
 <+\infty
\end{equation}
then
\[
 \pr \left ( \sum_{n=0}^\infty \prod_{i=0}^n \mathbf{t}_{i,n}^{F_i} <+\infty
\right )=1.
\]
In particular if Equation~\eqref{eq:sumprod} holds
\[
 \pr \left ( \sum_{n=0}^\infty \prod_{i=0}^n \mathbf{t}_{i,n}^{F_i} <+\infty \Big | \mathcal{N}=\mathbf{f}
\right )=1
\]
for $\theta$-almost all $\mathbf{f}=\{f_i\}_{i \in \N} \in {(S_X)^{\mathbb{N}}}$.
\end{lem}

\begin{proof} 
 Let $\xi:= \sum_{n=0}^\infty \prod_{i=0}^n \mathbf{t}_{i,n}^{F_i}$. Note that, by the Monotone Convergence Theorem and
using the independence of $\{F_i\}_{i \in \N}$,
\[
\begin{split}
 \E \left [ \xi \right ] &=  \E_\theta \left [ \xi \right ] 
= \sum_{n=0}^\infty \E_\mu \left [ \prod_{i=0}^n \mathbf{t}_{i,n}^{F_i} \right ] 
= \sum_{n=0}^\infty \prod_{i=0}^n G_{F_i}(\mathbf{t}_{i,n}) < +\infty
\end{split}
\]
thus $\pr(\xi<+\infty)\equiv\theta(\xi <+\infty)=1$. The last equality in the statement follows easily from Equation~\eqref{eq:disintegration}.
\end{proof}

\begin{proof}[Proof of Theorem~\ref{thm:3}]
    We study the behavior of the single-station counterpart and we initially follow the proof of \cite[Proposition 3.1]{cf:JMZ}.
Note that $\pr(\widetilde R_{i}<n-i+1)=G_\boldmu(\mathbf{\mathfrak{r}}_i(n-i+1)|i)$ is the probability that the
$n+1$-vertex does not belong to the radius of influence of the $i$-th vertex.
Hence $\prod_{i=0}^n G_\boldmu(\mathbf{\mathfrak{r}}_i(n-i+1)|i)$ is the probability
that the $n+1$-vertex does not belong to the radius of influence of any vertex to its left.
Denote this event by $E_n$:
by Borel-Cantelli $\pr(\limsup_n E_n)=1$. Whence, there exists $n_0$ such that
for all $\pr \left (\bigcap_{k \ge n_0} E_k^\complement \right )>0$. Since $\pr(V_{n_0})>0$,
where $V_{n_0}$ is the event ``all the stations at $0,1,\ldots,n_0$ are activated'', we have
(using the FKG inequality)
\[
 \pr(V) \ge \pr \bigg (\bigcap_{k \ge n_0} E_k^\complement \Big | V_{n_0} \bigg ) \pr(V_{n_0})=
\pr \bigg ( \bigcap_{k \ge n_0} E_k^\complement \cap V_{n_0}
\bigg ) \ge \pr \bigg ( \bigcap_{k \ge n_0} E_k^\complement \bigg ) \pr(V_{n_0})>0.
\]
In particular if we have a deterministic environment, say $F_i:=f_i \in S_X$ for all $i \in \N$, where $f_0 \neq \mathbf{0}$ and
\begin{equation}\label{eq:sumprod4}
 \sum_{n=0}^\infty \prod_{i=0}^n \mathbf{\mathfrak{r}}_i(n-i+1)^{f_i}<+\infty
\end{equation}
then, since $G_{\boldmu}(\mathbf{z})=\mathbf{z}^{f_{i}}$, Equation~\eqref{eq:sumprod3} holds and $\pr(V)>0$.

As in the previous proofs, we condition on the event $\{F_0 \neq \mathbf{0}\}$.
Finally, by Lemma~\ref{lem:1} (using $\mathbf{t}_{i,n}:= \mathbf{\mathfrak{r}}_i(n-i+1)$) we see that Equation~\eqref{eq:sumprod4}
holds for $\widetilde \theta$-almost all configurations $\mathbf{f} \in \widehat{{(S_X)^{\mathbb{N}}}}$ and this yields the result.
\end{proof}

  \begin{proof}[Proof of Theorem~\ref{th:reverseinhomogeneous}]
   If we define
\[
 \xi_{n}:= \sum_{k \ge 1} (1-\mathbf{\mathfrak{r}}_{n+k}(k)^{F_{n+k}}), \qquad
\zeta:= \sum_{n \in \N} \prod_{k=1}^\infty \mathbf{\mathfrak{r}}_{n+k}(k)^{F_{n+k}}
\]
then, by the Monotone Convergence Theorem and the  Bounded Convergence Theorem,
\[
 \E_\mu[\xi_{n}]=\sum_{k \ge 1} (1-G_\boldmu(\mathbf{\mathfrak{r}}_{n+k}(k)|n+k)),
\qquad
\E_\mu[\zeta]=\sum_{n \in \N} \prod_{k=1}^\infty G_\boldmu(\mathbf{\mathfrak{r}}_{n+k}(k)|n+k).
\]
We use the same notation of the proof of Theorem~\ref{thm:2}. According to \cite[Theorem 2.4(i)]{cf:JMZ}, $\E_\mu[\xi_{n}]=+\infty$ if and only if $\pr(S)=1$
(almost sure survival of the single-station counterpart, if there is a station at the origin with probability one) which is
equivalent to $\pr(S|F_0 \neq \mathbf{0})=1-\mu_0(\mathbf{0})$ which, in turn, is equivalent to $\pr(S|\mathcal{N}=\mathbf{f})
=1$ for $\widehat \theta$-almost all configurations $\mathbf{f} \in \widehat{{(S_X)^{\mathbb{N}}}}$. Moreover,
$\E_\mu[\zeta]<+\infty$ implies $\zeta <+\infty$ for $\widehat \theta$-almost all configurations  $\mathbf{f} \in \widehat{{(S_X)^{\mathbb{N}}}}$ and then,
according to \cite[Theorem 2.4(ii)]{cf:JMZ}, $\pr(S|\mathcal{N}=\mathbf{f})>0$ for $\theta$-almost all configurations $\mathbf{f}$.
\end{proof}

The following lemma, basically known as Kummer's convergence test, is required in Remark~\ref{rem:kummer} and short proofs can be found in \cite{cf:Sjodin, cf:tong}.

\begin{lem}
   \label{lem:kummer} 
   Let $\{a_n\}_{n \in \N}$ be a sequence of positive numbers. The following are equivalent.
   \begin{enumerate}
   \item $\sum_{n \in \N}a_n <+\infty$;
       \item there exist $\alpha >0$, $N \in \N$, and a sequence $\{B_n\}_{n \ge N}$ of positive numbers such that $B_n a_n/a_{n+1}-B_{n+1} \ge \alpha$ for all $n \ge N$;
       \item there exist $\alpha >0$, $N \in \N$, and a sequence $\{B_n\}_{n \ge N}$ of positive numbers such that $B_n a_n/a_{n+1}-B_{n+1} = \alpha$ for all $n \ge N$;
       \item there exist $\alpha >0$, $N \in \N$ and $B_N>0$  such that $B_{n+1}:= B_n a_n/a_{n+1} - \alpha$ for all $n \ge N$ defines recursively a sequence of positive numbers;
       \item there exist $\alpha >0$, $N \in \N$ and $B_N>0$  such that $B_{n+1}:= \big (B_N a_N - \alpha \sum_{i=N+1}^{n+1} a_i \big )/a_{n+1}>0$ for all $n \ge N$.
   \end{enumerate}
Moreover $\sum_{n \in \N}a_n =+\infty$ if and only if there exists a sequence $\{B_n\}_{n \in \N}$ of positive numbers such that $\sum_{n \in \N} 1/B_n=+\infty$ and $B_n (a_n/a_{n+1}) \le B_{n+1}$.
\end{lem}

\begin{proof}[Proof of Theorem~\ref{pro:comparehomfirework}]
 Recall that under Assumption~\ref{ass:1} or Assumption~\ref{ass:2}, there exists $T>0$ such that for all $t \ge T$ and $y \in Y$, we have Equation~\eqref{eq:assgerm}, that is,
$    G_\boldmu(\mathbf{\mathfrak{r}}_y(t)|y) 
\le G_\boldnu(\mathbf{\overline{\mathfrak{r}}}_y(t)|y)$.

Observe that, for every $n_0 \in \N$,
\[
 \sum_{n=0}^\infty \prod_{i=0}^n G_{\boldmu}(\bm{\mathfrak{r}}(i+1)|0)=+\infty
 \Leftrightarrow
  \sum_{n=n_0}^\infty \prod_{i=0}^n G_{\boldmu}(\bm{\mathfrak{r}}(i+1)|0)=+\infty
   \Leftrightarrow
  \sum_{n=n_0}^\infty \prod_{i=n_0}^n G_{\boldmu}(\bm{\mathfrak{r}}(i+1)|0)=+\infty
\]
since $G_{\boldmu}(\bm{\mathfrak{r}}(i+1)|0)>0$ for all $i \in \N$. 

\begin{enumerate}
\item 
Take $n_0 \ge T$; applying Theorem~\ref{thm:2} and Equations~\eqref{eq:betterthm2} and \eqref{eq:assgerm}, we have
\[
\begin{split}
    \theta_\boldnu(\mathbf{f}\colon \pr_\boldnu(V|&\mathcal{N}_\boldnu=\mathbf{f})>0)=1-\nu(\mathbf{0}) 
    \Longrightarrow 
     \sum_{n=0}^\infty \prod_{i=0}^n G_{\boldnu}(\bm{\mathfrak{r}}(i+1)|0)<+\infty\\
 & 
\Longrightarrow
    \sum_{n=n_0}^\infty \prod_{i=n_0}^n G_{\boldnu}(\bm{\mathfrak{r}}(i+1)|0)<+\infty
\Longrightarrow
    \sum_{n=n_0}^\infty \prod_{i=n_0}^n G_{\boldmu}(\bm{\mathfrak{r}}(i+1)|0)<+\infty\\
   &\Longrightarrow
  \sum_{n=0}^\infty \prod_{i=0}^n G_{\boldmu}(\bm{\mathfrak{r}}(i+1)|0)<+\infty 
    \Longrightarrow
    \theta_\boldmu(\mathbf{f}\colon \pr_\boldmu(V|\mathcal{N}_\boldmu=\mathbf{f})>0)=1-\mu(\mathbf{0})
\end{split}
\]

\item Similarly, take $n_0 \ge T$; applying Theorem~\ref{thm:2} and Equations~\eqref{eq:betterthm2} and \eqref{eq:assgerm}, we have
\[
\begin{split}
    \theta_\boldmu(\mathbf{f}\colon \pr_\boldmu(V|&\mathcal{N}_\boldmu=\mathbf{f})=0)=1 
    \Longrightarrow 
     \sum_{n=0}^\infty \prod_{i=0}^n G_{\boldmu}(\bm{\mathfrak{r}}(i+1)|0)=+\infty\\
 & 
\Longrightarrow
    \sum_{n=n_0}^\infty \prod_{i=n_0}^n G_{\boldmu}(\bm{\mathfrak{r}}(i+1)|0)=+\infty
\Longrightarrow
    \sum_{n=n_0}^\infty \prod_{i=n_0}^n G_{\boldnu}(\bm{\overline{\mathfrak{r}}}(i+1)|0)=+\infty\\
   &\Longrightarrow
  \sum_{n=0}^\infty \prod_{i=0}^n G_{\boldnu}(\bm{\overline{\mathfrak{r}}}(i+1)|0)=+\infty 
    \Longrightarrow
    \theta_\boldnu(\mathbf{f}\colon \pr_\boldnu(V|\mathcal{N}_\boldnu=\mathbf{f})=0)=1
\end{split}
\]

\end{enumerate}

  \end{proof}

  \begin{proof}[Proof of Theorem~\ref{pro:comparereverse}]
  As in the previous proof, we note that, for every $n_0 \in \N$
  \[
  \sum_{i \in \N} (1-G_\boldmu(\mathfrak{r}_{n+i}(i)|n+i))=+\infty \Leftrightarrow 
  \sum_{i =n_0} (1-G_\boldmu(\mathfrak{r}_{n+i}(i)|n+i))=+\infty.
  \]
  \begin{enumerate}
      \item Take $n_0 \ge T$; applying Theorem~\ref{th:reverseinhomogeneous} and Equations~\eqref{eq:betterthm3} and \eqref{eq:assgerm}, we have
      \[
      \begin{split}
\theta_\boldnu(\mathbf{f}\colon& \pr_\boldnu(S|\mathcal{N}_\boldnu=\mathbf{f})=1)=1-\nu_0(\mathbf{0}) \Longrightarrow \sum_{i \in \N} (1-G_\boldnu(\overline{\mathfrak{r}}_{n+i}(i)|n+i))=+\infty, \, \forall n \in \N \\
& \Longrightarrow 
\sum_{i =n_0} (1-G_\boldnu(\overline{\mathfrak{r}}_{n+i}(i)|n+i))=+\infty, \, \forall n \in \N  \\
&
\Longrightarrow 
\sum_{i =n_0} (1-G_\boldmu(\mathfrak{r}_{n+i}(i)|n+i))=+\infty, \, \forall n \in \N  \\
& \Longrightarrow \sum_{i \in \N} (1-G_\boldmu(\mathfrak{r}_{n+i}(i)|n+i))=+\infty, \, \forall n \in \N  \Longrightarrow
 \theta_\boldmu(\mathbf{f}\colon \pr_\boldmu(S|\mathcal{N}_\boldmu=\mathbf{f})=1)=1-\mu_0(\mathbf{0}).
\end{split}
      \]
      \item Similarly, take $n_0 \ge T$; applying Theorem~\ref{th:reversehomogeneous} and Equations~\eqref{eq:betterthm3} and \eqref{eq:assgerm}, we have
      \[
      \begin{split}
\theta_\boldmu(\mathbf{f}\colon \pr_\boldmu(S|&\mathcal{N}_\boldmu=\mathbf{f})=0)=1 \Longrightarrow \sum_{i \in \N} (1-G_\boldmu({\mathfrak{r}}(i)|1))<+\infty \\
& \Longrightarrow 
\sum_{i =n_0} (1-G_\boldmu({\mathfrak{r}}(i)|1))<+\infty 
\Longrightarrow 
\sum_{i =n_0} (1-G_\boldnu(\overline{\mathfrak{r}}(i)|1))<+\infty \\
& \Longrightarrow \sum_{i \in \N} (1-G_\boldnu(\overline{\mathfrak{r}}(i)|1))<+\infty \Longrightarrow
 \theta_\boldnu(\mathbf{f}\colon \pr_\boldnu(S|\mathcal{N}_\boldnu=\mathbf{f})=1)=1.
 \end{split}
    \]  
  \end{enumerate}
      
  \end{proof}

\section*{Acknowledgement}

The authors would like to thank the anonymous Referees for their careful reading of the paper and their valuable suggestions.

\bibliography{bibliography}{}

\begin{thebibliography}{10}

\bibitem{cf:Agresti75}
Alan Agresti.
\newblock On the extinction times of varying and random environment branching
  processes.
\newblock {\em J. Appl. Probability}, 12:39--46, 1975.

\bibitem{cf:BCZ2024}
Daniela Bertacchi, Elisabetta Candellero, and Fabio Zucca.
\newblock Martin boundaries and asymptotic behavior of branching random walks.
\newblock {\em Electronic Journal of Probability}, 29, 2024.

\bibitem{cf:BRZ16}
Daniela Bertacchi, Pablo~M. Rodriguez, and Fabio Zucca.
\newblock Galton-{W}atson processes in varying environment and accessibility
  percolation.
\newblock {\em Braz. J. Probab. Stat.}, 34(3):613--628, 2020.

\bibitem{cf:BZ2}
Daniela Bertacchi and Fabio Zucca.
\newblock Characterization of critical values of branching random walks on
  weighted graphs through infinite-type branching processes.
\newblock {\em J. Stat. Phys.}, 134(1):53--65, 2009.

\bibitem{cf:BZ4}
Daniela Bertacchi and Fabio Zucca.
\newblock Recent results on branching random walks.
\newblock In {\em Statistical Mechanics and Random Walks: Principles, Processes
  and Applications}, pages 289--340. Nova Publisher, 2013.

\bibitem{cf:BZ13}
Daniela Bertacchi and Fabio Zucca.
\newblock Rumor processes in random environment on {$\mathbb{Z}^d$} and on
  {G}alton-{W}atson trees.
\newblock {\em J. Stat. Phys.}, 153(3):486--511, 2013.

\bibitem{cf:BZ14-SLS}
Daniela Bertacchi and Fabio Zucca.
\newblock Strong local survival of branching random walks is not monotone.
\newblock {\em Adv. Appl. Probab.}, 46(4):400--421, 2014.

\bibitem{cf:BZ15}
Daniela Bertacchi and Fabio Zucca.
\newblock Branching random walks and multi-type contact-processes on the
  percolation cluster of {$\mathbb{Z}^d$}.
\newblock {\em Ann. Appl. Probab.}, 25(4):1993--2012, 2015.

\bibitem{cf:BZ2017}
Daniela Bertacchi and Fabio Zucca.
\newblock A generating function approach to branching random walks.
\newblock {\em Braz. J. Probab. Stat.}, 31(2):229--253, 2017.

\bibitem{cf:BZ2020}
Daniela Bertacchi and Fabio Zucca.
\newblock Branching random walks with uncountably many extinction probability
  vectors.
\newblock {\em Braz. J. Probab. Stat.}, 34(2):426--438, 2020.

\bibitem{cf:BZgerm}
Daniela Bertacchi and Fabio Zucca.
\newblock Strong survival and extinction for multitype branching processes
  walks via a new order for generating functions, 2024.
\newblock arXiv:2403.01565.

\bibitem{cf:BM08}
Erik Broman and Ronald Meester.
\newblock Survival of inhomogeneous {G}alton-{W}atson processes.
\newblock {\em Adv. in Appl. Probab.}, 40(3):798--814, 2008.

\bibitem{cf:Bigg}
J.~C. D'Souza and J.~D. Biggins.
\newblock The supercritical {G}alton-{W}atson process in varying environments.
\newblock {\em Stochastic Process. Appl.}, 42(1):39--47, 1992.

\bibitem{cf:Faern}
Dean~H. Fearn.
\newblock Galton-{W}atson processes with generation dependence.
\newblock In {\em Proceedings of the {S}ixth {B}erkeley {S}ymposium on
  {M}athematical {S}tatistics and {P}robability ({U}niv. {C}alifornia,
  {B}erkeley, {C}alif., 1970/1971), {V}ol. {IV}: {B}iology and health}, pages
  159--172. Univ. California Press, Berkeley, CA, 1972.

\bibitem{cf:GMPV09}
Nina Gantert, Sebastian M\"uller, Serguei Popov, and Marina Vachkovskaia.
\newblock Survival of branching random walks in random environment.
\newblock {\em J. Theoret. Probab.}, 23(4):1002--1014, 2010.

\bibitem{cf:Hut2022}
Tom Hutchcroft.
\newblock Transience and recurrence of sets for branching random walk via
  non-standard stochastic orders.
\newblock {\em Ann. Inst. Henri Poincar\'e{} Probab. Stat.}, 58(2):1041--1051,
  2022.

\bibitem{cf:disp3}
Valdivino~V. Junior, F\'abio~P. Machado, and Alejandro Rold\'an-Correa.
\newblock Uniform dispersion in growth models on homogeneous trees.
\newblock {\em ALEA Lat. Am. J. Probab. Math. Stat.}, 22(1):607--626, 2025.

\bibitem{cf:JMZ}
Valdivino~V. Junior, Fabio~P. Machado, and Mauricio Zuluaga.
\newblock Rumor processes on {$\mathbb N$}.
\newblock {\em J. Appl. Probab.}, 48(3):624--636, 2011.

\bibitem{cf:disp1}
Valdivino~Vargas Junior, F\'abio~Prates Machado, and Alejandro Rold\'an-Correa.
\newblock Dispersion as a survival strategy.
\newblock {\em J. Stat. Phys.}, 164(4):937--951, 2016.

\bibitem{cf:disp2}
Valdivino~Vargas Junior, F\'abio~Prates Machado, and Alejandro Rold\'an-Correa.
\newblock Evaluating dispersion strategies in growth models subject to
  geometric catastrophes.
\newblock {\em J. Stat. Phys.}, 183(2):Paper No. 30, 15, 2021.

\bibitem{cf:Ligg1}
Thomas~M. Liggett.
\newblock Branching random walks and contact processes on homogeneous trees.
\newblock {\em Probab. Theory Related Fields}, 106(4):495--519, 1996.

\bibitem{cf:MachadoMenshikovPopov}
F.~P. Machado, M.~V. Menshikov, and S.~Yu. Popov.
\newblock Recurrence and transience of multitype branching random walks.
\newblock {\em Stochastic Process. Appl.}, 91(1):21--37, 2001.

\bibitem{cf:MP03}
F.~P. Machado and S.~Yu. Popov.
\newblock Branching random walk in random environment on trees.
\newblock {\em Stochastic Process. Appl.}, 106(1):95--106, 2003.

\bibitem{cf:MenshikovVolkov}
M.~V. Menshikov and S.~E. Volkov.
\newblock Branching {M}arkov chains: qualitative characteristics.
\newblock {\em Markov Process. Related Fields}, 3(2):225--241, 1997.

\bibitem{cf:PemStac1}
Robin Pemantle and Alan~M. Stacey.
\newblock The branching random walk and contact process on {G}alton-{W}atson
  and nonhomogeneous trees.
\newblock {\em Ann. Probab.}, 29(4):1563--1590, 2001.

\bibitem{cf:dispSchi}
Rinaldo~B. Schinazi.
\newblock Does random dispersion help survival?
\newblock {\em J. Stat. Phys.}, 159(1):101--107, 2015.

\bibitem{cf:Sjodin}
Tord Sjödin.
\newblock A short and unified proof of {K}ummer's test, 2018.
\newblock arXiv:1802.09858.

\bibitem{cf:tong}
Jing~Cheng Tong.
\newblock Kummer's test gives characterizations for convergence or divergence
  of all positive series.
\newblock {\em Amer. Math. Monthly}, 101(5):450--452, 1994.

\bibitem{cf:Z1}
Fabio Zucca.
\newblock Survival, extinction and approximation of discrete-time branching
  random walks.
\newblock {\em J. Stat. Phys.}, 142(4):726--753, 2011.

\end{thebibliography}
\bibliographystyle{plain}

\end{document}